\documentclass[a4paper,12pt,leqno, twoside]{article}
\usepackage{amssymb}
\usepackage{amsthm}
\usepackage[all]{xy} 
\usepackage[mathscr]{eucal}
\usepackage{hyperref}

\def\NM{{\mathbb{N}}}

\def\GM{{\mathbb{G}}}

\def\QM{{\mathbb{Q}}}
\def\FM{{\mathbb{F}}}
\def\ZM{{\mathbb{Z}}}

\def\ZG{{\mathfrak Z}}

\def\SC{{\mathcal S}}

\def\HC{{\mathcal H}}
\def\NC{{\mathcal N}}

\def\IC{{\mathcal I}}
\def\OC{{\mathcal O}}
\def\MC{{\mathcal M}}

\def\EC{{\mathcal E}}

\def\UC{{\mathcal U}}

\def\subsetneq{\varsubsetneq}

\def\simto{\buildrel\hbox{\tiny{$\sim$}}\over\longrightarrow}

\def\injo{\hookrightarrow}
\def\id{\mathop{\mathrm{Id}}\nolimits}

\def\ba{\backslash}
\def\wt{\widetilde}
\def\wh{\widehat}
\def\o#1{\overline{#1}}

\def\application#1#2#3#4#5{\begin{array}{rcl}
                            #1 \;\;\; #2 & \to &  #3 \\
                              #4 & \mapsto & #5 
                            \end{array}}

\def\To#1{\buildrel\hbox{\tiny{$#1$}}\over\longrightarrow}
\def\to{\rightarrow}

\def\Mod#1{\mathop{\hbox {\sl #1-Mod}}}     

\def\ker{{\rm ker}}

\def\im{{\rm im}}

\def\Hom{\mathop{\hbox{\rm Hom}}\nolimits}
\def\Aut{\mathop{\hbox{\rm Aut}}\nolimits}

\def\Mod{{\rm Mod}}
\def\Rep{{\rm {R}ep}}

\def\Irr#1#2{\mathop{{\rm Irr}_{#1}\left(#2\right)}}

\def\Ind#1#2{\hbox {\rm Ind}_{#1}^{#2}}

\def\cind#1#2#3{\hbox {\rm ind}_{#1}^{#2}\>\!\left(#3\right)} 
\def\ip#1#2#3{\hbox {\sl i}_{#1}^{#2}\>\!(#3)}  

\def\dim{\mathop{\mbox{\rm dim}}\nolimits}
\def\dim{{\rm dim}}

\def\ini{\setcounter{equation}{\value{subsubsection}}\addtocounter{subsubsection}{1}}

\makeatletter
\renewcommand{\subsubsection}{\@startsection{subsubsection}{3}{\parindent}{-\baselineskip}{-0.01\baselineskip}{\bf}}
\renewcommand*{\@seccntformat}[1]{%
  \csname the#1\endcsname\
}
\makeatother  

\def\ali{\subsubsection{}\setcounter{equation}{0}}
\def\alin#1{\setcounter{equation}{0}\subsubsection{\it  #1}. --- }

\newtheoremstyle{th}
  {\baselineskip}{.5\baselineskip}{\itshape}
  {\parindent}{\bf}
  { ---}{.5em}{}

\newtheoremstyle{def}
  {\baselineskip}{\baselineskip}{}
  {\parindent}{\bf}
  {--}{.5em}{}

\newtheoremstyle{th*}
  {.5\baselineskip}{.5\baselineskip}{\itshape}
  {\parindent}{\bf}
  { ---}{.5em}{}

\newtheoremstyle{remark*}
  {.5\baselineskip}{.5\baselineskip}{}
  {\parindent}{\bf}
  {--}{.5em}{}

\newtheoremstyle{remark}
  {.5\baselineskip}{.5\baselineskip}{}
  {\parindent}{\bf}
  { ---}{.5em}{}

\swapnumbers
\theoremstyle{th}

\newtheorem{lemme}[subsubsection]{\it Lemma.\bf}
\newtheorem{prop}[subsubsection]{\it Proposition.\bf}
\newtheorem{coro}[subsubsection]{\it Corollary.\bf}

\theoremstyle{def}

\newtheorem{DEf}[subsubsection]{\it D{e}finition.\bf}

\theoremstyle{remark}
\newtheorem{rema}[subsubsection]{\it Remark.\bf}  
\newtheorem{exam}[subsubsection]{\it Example.\bf}   

\theoremstyle{th*}
\newtheorem*{thm}{\it Theorem.}
\newtheorem*{lem}{\it Lemma.}

\newtheorem*{cor}{\it Corollary.}
\newtheorem*{con}{\it Conjecture.}
\newtheorem*{expec}{\it Expectation.}

\theoremstyle{remark*}
\newtheorem*{rak}{\it Remark.}

\newcommand{\findem}{\hfill$\Box$\par\medskip}
\newcommand{\dem}{\indent {\it Preuve :} \rm }

\title{A functoriality principle for blocks of $p$-adic linear groups}
\author{Jean-Fran\c{c}ois Dat}
\date{}

\AtEndDocument{{\footnotesize
  \textsc{Jean-Fran\c{c}ois Dat.}  Institut de Math\'ematiques de Jussieu-PRG, Universit\'e Pierre et Marie Curie, 4
    place Jussieu, 75005 Paris, France.
  \textit{e-mail :} \texttt{jean-francois.dat@imj-prg.fr}
}}

\begin{document}
\maketitle
\bibliographystyle{plain}

\def\la{\langle}
\def\ra{\rangle}
\def\knr{{\wh{K^{nr}}}}
\def\ka{\wh{K^{ca}}}

\def\dd{D_d^\times}
\def\mdro{\MC_{Dr,0}}
\def\mdrn{\MC_{Dr,n}}
\def\mdr{\MC_{Dr}}
\def\mlto{\MC_{LT,0}}
\def\mltn{\MC_{{\rm LT},n}}
\def\mltno{\MC_{{\rm LT},n}^{(0)}}
\def\mlt{\MC_{\rm LT}}
\def\mltK{\MC_{LT,K}}
\def\LJ{{\rm LJ}}
\def\JL{{\rm JL}}
\def\SL{{\rm SL}}
\def\GL{{\rm GL}}
\def\PGL{{\rm PGL}}

\def\Ql{\QM_{\ell}}
\def\Zl{\ZM_{\ell}}
\def\Fl{\FM_{\ell}}
\def\oQl{\o\QM_{\ell}}
\def\oZl{\o\ZM_{\ell}}
\def\bZl{\o\ZM_{\ell}}
\def\bQl{\o\QM_{\ell}}
\def\oFl{\o\FM_{\ell}}
\def\mltnc{\wh\MC_{{\tiny{\rm LT}},n}}
\def\mltnco{\wh\MC_{{\tiny{\rm LT}},n}^{(0)}}

\def\sp{{\rm sp}}
\def\Ens{{\SC ets}}
\def\Coef{{\rm Coef}}

\begin{abstract}
Bernstein blocks of complex representations of $p$-adic reductive
groups have been computed in a large number of examples, in part
thanks to the theory of types \`a la Bushnell and Kutzko. The output of
these purely representation-theoretic computations is that
 many of these blocks are equivalent. The motto of this paper is that
 most of these coincidences are explained, and many more can be
 predicted, by a functoriality principle involving dual groups.
We prove a precise statement for groups related to $GL_{n}$, and then
state conjectural generalizations in two directions: more general
reductive groups and/or  \emph{integral} $l$-adic representations.
\end{abstract}

\tableofcontents
\section{Main statements}

Let $F$ be a $p$-adic field and let $R$ be a commutative ring in which $p$ is invertible. 
For $\mathbf{G}$ a reductive group
over $F$, we put $G:=\mathbf{G}(F)$ and  we denote by
$\Rep_{R}(G)$ the abelian category of smooth 
 representations of $G$ with coefficients in $R$, and by $\Irr{R}{G}$ the set of
 isomorphism classes of simple objects in $\Rep_{R}(G)$. We will be
 mainly interested in the cases $R=\oQl$ or $\oZl$ or $\oFl$ for
 $\ell$ a prime number coprime to $p$.

Let us assume $R=\oQl$ for a while. For a general $\mathbf{G}$,
Bernstein \cite{bernstein} has decomposed $\Rep_{\oQl} (G)$ as a
product  
of indecomposable abelian subcategories called \emph{blocks}. 
This decomposition is  characterized by the property that two
irreducible representations belong to the same Bernstein block if 
and only if their supercuspidal supports are ``inertially equivalent''.

When $\mathbf{G}=\GL_{n}$, Bushnell and Kutzko \cite{BK} \cite{BKsst} have proved that each block
 is equivalent to the category of modules over an  algebra of the
form $ \HC(n_{1},q^{f_{1}})\otimes_{\oQl}\cdots\otimes_{\oQl} \HC(n_{r},q^{f_{r}})$ where $\HC(n,q)$
denotes the extended Iwahori--Hecke algebra of type $A_{n-1}$ and parameter $q$ (the size of the
residual field of $F$). This shows in particular that, up to taking
tensor product of categories, all blocks of linear groups look ``the same''.
More precisely, joint with Borel's theorem, their work shows that any
Bernstein block is equivalent to the \emph{principal}
block  of a product of general linear groups.
Here, as usual, the ``principal block'' is the
one that contains the trivial representation. 

The main result of this paper is a ``Langlands-dual'' explanation of
this redundancy among blocks of linear groups. It will appear as a
functoriality principle for blocks, and also will point to
particularly nice equivalences, related to the usual functoriality
principle for irreducible representations.

However, the main interest of the paper probably lies in the
conjectural natural generalizations of the main
result. These speculations will
involve more general reductive groups $\mathbf{G}$ and/or coefficients
$R=\oZl$ or $\oFl$ for $\ell$ a prime number coprime to $p$.

\subsection{Functoriality for $\oQl$-blocks of groups of $\GL$-type}

We say that $\mathbf{G}$ is ``of GL-type'' if it is isomorphic to a product
of restriction of scalars of general linear groups over finite extensions of $F$.

\alin{Langlands parametrization of  $\oQl$-blocks} \label{Langparql}
Let
$^{L}\mathbf{G}= \hat\mathbf{G}\rtimes W_{F}$  be ``the'' dual group of $\mathbf{G}$, where $W_{F}$ is
the Weil group of $F$. For a general $\mathbf{G}$, 
Langlands' parametrization
conjecture predicts the existence of  a finite-to-one map $\pi\mapsto \varphi_{\pi}$
$$\Irr{\oQl}{G}\To{} \Phi(\mathbf{G},\oQl):=\{ \varphi : W_{F}\ltimes \oQl \To{}
{^{L}\mathbf{G}}(\oQl)\}_{/\sim}$$ 
where the right hand side denotes the set of ``admissible''\footnote{``admissible''
  implies in particular that the image of $W_{F}$ consists of semi-simple elements, an
  element of $^{L}\mathbf{G}$ being
  semi-simple if its projection on any algebraic quotient $\hat \mathbf{G}\rtimes \Gamma_{F'/F}$ is
  semi-simple.} 
$L$-parameters for $\mathbf{G}$ modulo
conjugacy by $\hat\mathbf{G}(\oQl)$. 
For $\mathbf{G}$ of $\GL$-type, this parametrization follows from
the Langlands correspondence of \cite{HaTay} \cite{HeLang} for
$\GL_{n}$  via the
non-commutative Shapiro lemma, and it is a \emph{bijection}.
Moreover, this correspondence is known to be compatible with parabolic induction in the following
sense. If $\pi$ is an irreducible subquotient of some induced representation $\ip{M,P}{G}\sigma$,
then ${\varphi_{\pi}}_{|W_{F}} \sim {^{L}\iota}\circ {\varphi_{\sigma}}_{|W_{F}}$, where $^{L}\iota :
{^{L}\mathbf{M}} \injo {^{L}\mathbf{G}}$ is any embedding dual to $\mathbf{M}\injo \mathbf{G}$
(note that $\mathbf{M}$ is also of $\GL$-type). As a consequence, for two irreducible
representations $\pi,\pi'$ in the same Bernstein block, we have ${\varphi_{\pi}}_{|I_{F}}\sim
{\varphi_{\pi'}}_{|I_{F}}$, 
hence we get a decomposition
$$ \Rep_{\oQl}(G) =\prod_{\phi\in \Phi_{\rm inert}(\mathbf{G},\oQl)} \Rep_{\phi}(G)$$
where the index set $\Phi_{\rm inert}(\mathbf{G},\oQl)=H^{1}(I_{F},\hat\mathbf{G}(\oQl))^{W_{F}}$ 
is the set of $\hat\mathbf{G}$-conjugacy classes of \emph{admissible inertial parameters}, i.e.
continuous\footnote{In all this discussion $\hat\mathbf{G}(\oQl)$ is equipped with the discrete topology.}
 sections 
$I_{F}\To{}{^{L}\mathbf{G}}(\oQl)$ that admit an extension to an admissible $L$-parameter in
$\Phi(\mathbf{G},\oQl)$, 
and $\Rep_{\phi}(G)$ consists of all smooth $\oQl$-representations of $G$, any 
irreducible subquotient $\pi$ of which  satisfies ${\varphi_{\pi}}_{|I_{F}}\sim\phi$.

It is ``well known''  that $\Rep_{\phi}(G)$ is actually a Bernstein block, see Lemma 
\ref{lemmeblockQl}, so that
we get a parametrization of blocks of $\Rep_{\oQl}(G)$ by $\Phi_{\rm inert}(\mathbf{G},\oQl)$. 
In this parametrization, the principal block corresponds to the trivial parameter 
$i \mapsto (1,i)$, which may explain why it is sometimes referred to as the ``unipotent'' block.

\alin{Functorial transfer of $\oQl$-blocks} \label{transferql}
Let $\mathbf{G'}$ be another group of $\GL$-type over $F$, and suppose given a 
morphism\footnote{Since we work with the Weil form of $L$-groups, we require  that a morphism of
  $L$-groups carries semi-simple elements to semi-simple elements.}
of
$L$-groups $\xi :\, {^{L}\mathbf{G}'}\To{} {^{L}\mathbf{G}}$. The composition map $\varphi'\mapsto
\xi\circ\varphi$ on $L$-parameters translates into a map $\Irr{\oQl}{G'}\To{\xi_{*}}\Irr{\oQl}{G}$,
known as the ``(local) Langlands' transfer induced by $\xi$''. 
A natural question to ask is whether this Langlands transfer can be
extended to non-irreducible representations in a functorial way. In
general there seems to be little  hope,
but the following result shows that it becomes possible under certain circumstances.

Let us fix an admissible inertial parameter $\phi':I_{F}\To{}{^{L}\mathbf{G}}$ and put
$\phi:=\xi\circ\phi'$. 
As usual, let $C_{\hat\mathbf{G}}(\phi)$ denote the centralizer in $\hat\mathbf{G}(\oQl)$ of the image
of the inertial parameter $\phi$.

\begin{thm}
Suppose that $\xi$ induces an isomorphism $C_{\hat\mathbf{G}'}(\phi')\simto C_{\hat\mathbf{G}}(\phi)$.
Then there is an 
equivalence of categories $\Rep_{\phi'}(G')\simto \Rep_{\phi}(G)$ that interpolates the
Langlands transfer $\xi_{*}$ on irreducible representations.
\end{thm}

\begin{rak}
  We expect that such an equivalence will be unique, up to natural
  transformation. In fact, the equivalences that we will exhibit are also compatible with parabolic
  induction,  and this extra compatibility makes them unique.
\end{rak}

\ali Let us give some examples that may shed light on the statement.
\label{examples}

i) Suppose that $\xi$ is a dual Levi embedding. This means that
$\mathbf{G'}$ may be embedded as a $F$-Levi subgroup of $\mathbf{G}$. Fix
a parabolic subgroup $P$ of $G$ with Levi component $G'$, and let
$[M',\pi']_{G'}$ be the inertial equivalence class of supercuspidal
pairs associated with  $\phi'$. Then the condition in the theorem is
equivalent to the requirement that the stabilizer of $[M',\pi']_{G'}$
in $G$ is $G'$. In this situation, it is well-known that the normalized parabolic
induction functor $i_{P}$ provides an equivalence of categories as in the
theorem.

ii) Suppose that $\xi$ is a base change homomorphism $\GL_{n}\times
W_{F }\To{} {^{L}({\rm Res}_{F'|F}}\GL_{n})$, and let $\phi'$ and
$\phi$ both be trivial. Then the condition in the theorem is met if and only
if $F'$ is totally ramified over $F$.
On the other hand, $\Rep_{\phi'}(G')$ is the principal block of
$\GL_{n}(F)$ while $\Rep_{\phi}(G)$ is the principal block of
$\GL_{n}(F')$. So Borel's theorem tells us they are respectively
 equivalent to the category of right modules over $\HC(n,q_{F})$, resp. $\HC(n,q_{F'})$. Therefore
they are equivalent if and only if $F'$ is totally ramified over
$F$. Moreover, it is not hard to find an equivalence that meets
the requirement of the theorem.

iii) Suppose that $\xi$ is an isomorphism of $L$-groups. Then the conditions of the theorem are met
for all $\phi'$ ! To describe the equivalence, first conjugate  under
$\hat\mathbf{G}$ to put  $\xi$ in the form $\hat\xi\times\psi$ where
$\hat\xi:\hat\mathbf{G}'\To{}\hat\mathbf{G}$ is an \'epinglage preserving $W_{F}$-equivariant
isomorphism, and $\psi : W_{F}\To{} Z(\hat \mathbf{G})\rtimes W_{F}$. Then $\hat\xi$ provides an $F$-rational
isomorphism $\hat\xi^{\vee}:\,\mathbf{G}\simto \mathbf{G'}$ (well-defined up to conjugacy), and $\psi$ determines a
character $\psi^{\vee} : G\To{} \mathbf{G}_{\rm
  ab}(F)\To{}\oQl^{\times}$ through local class field theory. The desired equivalence is
given by pre-composition under $\xi^{\vee}$ followed by  twisting under $\psi^{\vee}$. Its
compatibility with the Langlands transfer is Proposition 5.2.5 of \cite{Hainesstable}.

iv) Suppose that $\mathbf{G}=\GL_{n}$ and let $\phi$ be an inertial
parameter of $\mathbf{G}$ that is irreducible as a representation of $I_{F}$. Put
$\mathbf{G'}=\GL_{1}$ and $\phi'=1$ (trivial parameter).
Finally let $\xi$ be the product of  the central embedding
$\GL_{1}\injo\GL_{n}$ and  any extension $\varphi:\,
W_{F}\To{}\GL_{n}$ of $\phi$. Then the conditions of the Theorem are
met, $\Rep_{\phi}(G)$ is a cuspidal block and $\Rep_{\phi'}(G)$ is
$\Rep_{\oQl}(F^{\times}/\OC_{F}^{\times})$. The claimed equivalence is
given by $\chi\mapsto \pi\otimes (\chi\circ\det)$ where $\pi$
corresponds to $\varphi$ through Langlands' correspondence.

\alin{Reduction to unipotent blocks} \label{reduc_unip}
What makes Theorem \ref{transferql} a very flexible statement is that
no a priori restriction was made on $\xi$ ; namely its Galois
component can be very complicated. In this regard, it is important to
work with the Weil form of the $L$-group.
For instance, in the last example
above,  $\phi'$ was trivial but 
all the complexity of the setting was ``moved'' to the $L$-homomorphism $\xi$.
This is a simple example of a more general phenomenon that allows
to reduce Theorem \ref{transferql} to an equivalent statement which deals with a single
parameter $\phi\in\Phi_{\rm  inert}(\mathbf{G},\oQl)$. We now explain
this, and refer to section  \ref{sec:unip-fact} for details.

By definition we may choose an extension $\varphi$ of $\phi$ to an $L$-parameter $W_{F}\To{}
{^{L}\mathbf{G}(\oQl)}$. Conjugation by $\varphi(w)$ in
$^{L}\mathbf{G}$ then induces an action of
$W_{F}/I_{F}$ on $C_{\hat \mathbf{G}}(\phi)$ and a factorization :
$$\xymatrix{
\phi:\,\, I_{F} \ar[rrr]^-{i\mapsto (1,i)} &&&
 C_{\hat\mathbf{G}}(\phi)\rtimes W_{F} \ar[rrr]^-{(z,w)\mapsto z\varphi(w)} &&&
{^{L}\mathbf{G}(\oQl)} 
}.$$
It turns out that the outer action  $W_{F} \To{} {\rm Out}(C_{\hat\mathbf{G}}(\phi))$ does not depend on
the choice of $\varphi$ and thus defines a ``unique'' unramified group $\mathbf{G}_{\phi}$ over
$F$. Moreover, one checks that this group is of $\GL$-type, and
 that there exists $\varphi$ such that all $\varphi(w)$ preserve a
given \'epinglage of $C_{\hat\mathbf{G}}(\phi)$. So we get a factorization
of the form
$\phi: I_{F}\To{1\times \id} {^{L}\mathbf{G}_{\phi}}\To{\xi_{\varphi}}
{^{L}\mathbf{G}}$, as considered in 
Theorem \ref{transferql}. By construction, the hypothesis of the
latter theorem is satisfied, so we get  the following
\begin{cor}
  There is an equivalence of categories $\Rep_{1}({G}_{\phi})\simto\Rep_{\phi}(G)$ that
  extends the transfer associated to the above
  $\xi_{\varphi}:\,{^{L}\mathbf{G}}_{\phi}\injo {^{L}\mathbf{G}}$.  
\end{cor}

We like to see this statement as a ``moral explanation'' to, and a
more precise version of, 
the well known property that any $\oQl$-block of a
 $\GL_{n}$ is equivalent to the principal $\oQl$-block of a product of
linear groups over extension fields. Also we note that a different choice of extension $\varphi'$ of
$\phi$ as above will
differ from $\varphi$ by a central unramified cocycle $W_{F}\To{} Z(\hat\mathbf{G})$, and an associated
equivalence $\xi_{\varphi'}$ is thus deduced from $\xi_{\varphi}$ by twisting by the associated
unramified character of $G$.

\begin{rak}
  Again we may ask whether an equivalence as in the Corollary is actually unique. This
  boils down to the following concrete question. Let $\HC=\HC(n,q)$ be the extended affine Hecke
  algebra of type $A_{n}$ and parameter $q$. Suppose $\Phi$ is an auto-equivalence of
  categories of $\HC-\rm Mod$ that fixes all simple modules up to equivalence. Is
  $\Phi$ isomorphic to the identity functor ? Even more concretely : if $\IC$ is an
  invertible $\HC\otimes_{\oQl}\HC^{\rm opp}$-module such that $\IC\otimes_{\HC}S\simeq S$
  for any simple left $\HC$-module $S$, do we have $\IC\simeq \HC$ ? 
\end{rak}

\alin{On the proofs}
In fact, it is easy to show that Corollary \ref{reduc_unip} with $(\phi,\varphi)$
allowed to vary, is equivalent to Theorem
\ref{transferql} with $(\phi',\xi)$ allowed to vary, see Lemma \ref{equivstatements}. 
Now, to prove the corollary
directly,
 we may first reduce to the case
where $\mathbf{G}$ is quasi-simple, \emph{i.e.} of the form ${\rm Res}_{F'|F}\GL_{n}$, then
by a Shapiro-like argument to the case $\mathbf{G}=\GL_{n}$, then, using parabolic induction,
to the case where $\mathbf{G}_{\phi}$ is quasi-simple. In the latter case, $\Rep_{\phi}(\mathbf{G})$
is cut out by a simple type of \cite{BK} and the desired equivalence follows from the
computation of Hecke 
algebras there. Note that the information needed on the simple type is very coarse ; only the
degree and residual degree of the field entering the definition of the type is involved here. Details
are given in \ref{sec:groups-gl-type}.

\alin{Variants} \label{variant}
In the foregoing discussion, we may try to replace admissible parameters with domain
$I_{F}$ by admissible parameters with domain any closed normal subgroup $K_{F}$ of $I_{F}$. The main
examples we have in mind are $K_{F}=P_{F}$, the wild inertia subgroup, 
and $K_{F}=I_{F}^{(\ell)}:={\rm ker}(I_{F}\To{}\Zl(1))$, which is the maximal closed subgroup of
$I_{F}$ with prime-to-$\ell$ pro-order. Other possibilities are ramication subgroups of $P_{F}$.
In any case, an \emph{admissible $K_{F}$-parameter of $\mathbf{G}$} will be a morphism
$K_{F}\To{}{^{L}\mathbf{G}}$ that admits an extension to a usual admissible parameter $W_{F}\To{}{^{L}\mathbf{G}}$.
By grouping Bernstein blocks, we thus get a product decomposition 
$\Rep_{\oQl}(G)=\prod_{\phi} \Rep_{\phi}(G)$ where $\phi$ runs over admissible $K_{F}$-parameters up
to $\hat \mathbf{G}$ conjugacy, and $\Rep_{\phi}(G)$ is ``generated'' by the irreducible representations $\pi$
such that $(\varphi_{\pi})_{|K_{F}}\sim \phi$. For example, in the case $K_{F}=P_{F}$, the factor
$\Rep_{1}(G)$ is the level $0$ subcategory.

It is then easy to deduce from Theorem \ref{transferql} exactly the same statement
 for $K_{F}$-parameters, simply by grouping the equivalences provided by this theorem.
In contrast, our arguments in this paper, in particular in paragraph \ref{computGL},
 allow us to prove 
the natural analogue of Corollary \ref{reduc_unip} only when $K_{F}$ contains $P_{F}$.

\subsection{Functoriality for $\oZl$-blocks of groups of $\GL$-type}

For $\mathbf{G}$ of $\GL$-type, Vign\'eras \cite{VigInduced} has obtained a decomposition of $\Rep_{\oFl}(G)$ formally
analogous to that of Bernstein,  where the blocks are indexed by inertial classes of supercuspidal
pairs over $\oFl$.
This was further lifted to a decomposition of $\Rep_{\oZl}(G)$ by
Helm in \cite{HelmBernstein}.
In general, Vign\'eras or Helm blocks will \emph{not} be equivalent to categories of modules over an
Hecke--Iwahori algebra, and actually even the structure of the principal block of
$\Rep_{\oZl}(G)$ (which may contain non-Iwahori-spherical representations) is not yet well understood. 

\alin{Langlands parametrization of $\oZl$-blocks} \label{Lanparzl}
In exactly the same way as for $\oQl$-blocks (see \ref{lemmeblockZl} for some details), Vign\'eras' Langlands correspondence for
$\oFl$-representations \cite{VigLanglands} 
allows one to rewrite the Vign\'eras--Helm decomposition in the form 
$$ \Rep_{\oZl}(\GL_{n}(F)) =\prod_{\bar\phi\in \Phi_{\rm inert}(\GL_{n},\oFl)} \Rep_{\bar\phi}(\GL_{n}(F))$$
with 
$\Phi_{\rm inert}(\GL_{n},\oFl)$ 
the set of equivalence classes of semi-simple $n$-dimensional $\oFl$-representations of
$I_{F}$ that extend to $W_{F}$.
This suggests to use $L$-groups over $\oFl$ in order to mimic the transfer of $\oQl$-blocks. 
 \emph{However,} in order to get a functoriality property analogous to Theorem \ref{transferql}, we
need to stick to the usual $L$-groups over $\oQl$.

Recall that we have a ``semisimplified reduction mod $\ell$'' map 
$r_{\ell}:\, \Phi_{\rm inert}(\GL_{n},\oQl)\To{} \Phi_{\rm inert}(\GL_{n},\oFl)$.  The
basic properties of the Vign\'eras--Helm decomposition tell us 
that, denoting by $e_{\bar\phi}$ the primitive idempotent in the center $\ZG_{\oZl}(G)$ of
$\Rep_{\oZl}(G)$  that
cuts out the block $\Rep_{\bar\phi}(G)$,  we have the equality
$$ e_{\bar\phi}=\sum_{r_{\ell}(\phi)=\bar\phi} e_{\phi},
\;\;\; \hbox{ in } \;\;\;
\ZG_{\oQl}(G).$$

Now recall the ``prime-to-$\ell$ inertia subgroup'' $I_{F}^{(\ell)}$ of \ref{variant}, and
define $ \Phi_{\ell'\rm -inert}(\GL_{n},R)$ as the set of semi-simple $n$-dimensional
$R$-representations of $I_{F}^{(\ell)}$ that extend to $W_{F}$ (here, $R$ is either $\oFl$ or $\oQl$).
We have a  commutative diagram
$$\xymatrix{
\Phi_{\rm inert}(\GL_{n},\oQl) \ar[r]^-{\rm res} \ar[d]_{r_{\ell}} &
\Phi_{\ell'\rm -inert}(\GL_{n},\oQl) \ar[d]^{r_{\ell}} \\
\Phi_{\rm inert}(\GL_{n},\oFl) \ar[r]_-{\rm res} &
\Phi_{\ell'\rm -inert}(\GL_{n},\oFl)
}.$$
The reduction map $r_{\ell}$ on the right hand side is a bijection because
  $I_{F}^{(\ell)}$ has prime-to-$\ell$ order. Moreover, 
the restriction map on the bottom is also a bijection because
 a semisimple $\oFl$-representation of $I_{F}$ is determined by its
restriction to $I_{F}^{(\ell)}$ (indeed, it is determined by its Brauer character on
$\ell'$-elements, but the set of $\ell'$ elements of $I_{F}$ is precisely $I_{F}^{(\ell)}$).


This shows that we may parametrize the Vign\'eras-Helm blocks by the set 
$\Phi_{\ell'\rm -inert}(\mathbf{G},\oQl)$, with the restriction map from $I_{F}$ to $I_{F}^{(\ell)}$
playing the role of the reduction map $r_{\ell}$. Using the non-commutative Shapiro lemma
(see Corollary \ref{coroShapiro}),
we get for any group $\mathbf{G}$ of $\GL$-type a parametrization of blocks of
$\Rep_{\oZl}(G)$ by the set 
$$ \Phi_{\ell'\rm -inert}(\mathbf{G},\oQl):=\left\{\phi_{\ell}:\,I_{F}^{(\ell)}\To{} {^{L}\mathbf{G}(\oQl)},\,
\exists \varphi\in \Phi(\mathbf{G},\oQl), \varphi_{|I_{F}^{(\ell)}}=\phi_{\ell}\right\}_{/
\hat\mathbf{G}-{\rm conj}}.$$

  

\alin{Functorial transfer of  $\oZl$-blocks} \label{transferzl}
Now consider again  an $L$-homomorphism $\xi:{^{L}\mathbf{G'}}\To{}{^{L}\mathbf{G}}$ of groups of
$\GL$-type,  fix an admissible parameter $\phi': I_{F}^{(\ell)}\To{}
{^{L}\mathbf{G'}}$ and put $\phi:=\xi\circ\phi'$. Attached to $\phi$ is a $\oZl$-block $\Rep_{\phi,\oZl}(G)$, whose $\oQl$-objects
form a finite sum of
$\oQl$-blocks $\Rep_{\phi,\oQl}(G)=\prod_{\psi_{|I_{F}^{(\ell)}}\sim \phi} \Rep_{\psi}(G)$. 

\begin{con}
  Suppose that $\xi$ induces an isomorphism $C_{\hat \mathbf{G}'}(\phi')\simto C_{\hat
    \mathbf{G}}(\phi)$, and also that the projection of $\xi(W_{F})$ to $\hat
  G(\oQl)$ is bounded. 
Then there is an 
equivalence of categories $\Rep_{\phi',\oZl}(G')\simto \Rep_{\phi,\oZl}(G)$ that interpolates the
Langlands transfer $\xi_{*}$ on irreducible $\oQl$-representations.
\end{con}

Again we may also conjecture that there is a unique such equivalence of categories, or at
least that any equivalence $\Rep_{\phi',\oQl}(G')\simto\Rep_{\phi,\oQl}(G)$
   as in Theorem \ref{transferql} extends to an equivalence
 $\Rep_{\phi',\oZl}(G')\simto\Rep_{\phi,\oZl}(G)$.

\begin{rak}
Let us take up the base change  example  \ref{examples} ii). With the notation there, the
requirements of the conjecture are met if and only if $F'$ is a totally ramified extension
of \emph{degree prime to $\ell$.} The conjecture then predicts an
equivalence between the principal $\oZl$-blocks of $\GL_{n}(F)$ and $\GL_{n}(F')$. In fact, it is
plausible that such an equivalence exists when $F'$ is \emph{only} assumed to be totally
ramified, but in general it won't interpolate the base change of irreducible
$\oQl$-representations. As an example, put $n=2$, $F=\QM_{p}$,
$\ell|(p+1)$ odd, and
$F'=\QM_{p}(p^{1/\ell})$. In this situation there exists a supercuspidal
$\oQl$-representation $\pi$ in the principal $\oZl$-block of
$\GL_{n}(F)$ whose base change is a principal series of $\GL_{n}(F')$.
Indeed, take for $\pi$ the representation that corresponds to the
irreducible Weil group representation $\sigma:={\rm
  ind}_{W_{\QM_{p^{2}}}}^{W_{\QM_{p}}}(\chi)$ where $\chi$ is any
character of $W_{\QM_{p^{2}}}\To{}\oZl^{\times}$ that extends a character
$I_{\QM_{p}}\twoheadrightarrow \mu_{\ell}\injo\oQl^{\times}$.
\end{rak}


\alin{Reduction to unipotent $\oZl$-blocks} \label{reduc_unipzl} Start with an admissible parameter $\phi:\,
I_{F}^{(\ell)}\To{} {^{L}\mathbf{G}}$ and choose an extension to some usual parameter 
$\varphi:\,W_{F}\To{} {^{L}\mathbf{G}}$. The same procedure as in paragraph \ref{reduc_unip}
provides us with a factorization
 $\phi: I_{F}^{(\ell)}\To{1\times \id} {^{L}\mathbf{G}_{\phi}}\To{\xi_{\varphi}}
{^{L}\mathbf{G}}$ in which $\mathbf{G}_{\phi}$ is a group of $\GL$-type that
 splits over a tamely ramified  $\ell$-extension, that
only depends on $\phi$,
 and such that
$\hat\mathbf{G}_{\phi}=C_{\hat\mathbf{G}}(\phi)$. In particular the assumption of the
last conjecture is satisfied and thus the following conjecture is a consequence of the
latter :
\begin{con}
    There is an equivalence of categories $\Rep_{1,\oZl}({G}_{\phi})\simto\Rep_{\phi,\oZl}(G)$ that
  extends the transfer of irreducible $\oQl$-representations  associated to the above
  $\xi_{\varphi}:\,{^{L}\mathbf{G}}_{\phi}\injo {^{L}\mathbf{G}}$.  
\end{con}

As in the case of $\oQl$ coefficients, Lemma \ref{equivstatements} tells us 
that conjectures \ref{reduc_unipzl} and \ref{transferzl} are actually
equivalent. 


\alin{Tame parameters and level $0$ blocks} \label{tamecase}
An $L$-homomorphism $\xi:{^{L}\mathbf{G'}}\To{}{^{L}\mathbf{G}}$ is called
\emph{tame} if its restriction $\xi_{|P_{F}}$ to
the wild inertia subgroup $P_{F}$ of $W_{F}$ is trivial (which means it is
$\hat\mathbf{G}$-conjugate to the map $p\mapsto (1,p)$). This 
definition also applies to $L$-parameters, for which $\mathbf{G'}$ is the trivial group, as
well as to inertial and $\ell$-inertial parameters. Note that neither $\mathbf{G'}$ nor
$\mathbf{G}$ are required to be tamely ramified.

If $\phi\in \Phi_{\ell'-\rm inert}(\mathbf{G},\oQl)$ 
is a tame $\ell$-inertial parameter of $\mathbf{G}$, 
the corresponding block $\Rep_{\phi}(G)$ in $\Rep_{\oZl}(G)$
has \emph{level} (or \emph{depth}) $0$, and conversely any level $0$ block of $\Rep_{\oZl}(G)$
 corresponds to a tame $\ell$-inertial
parameter. 
The following result is our  best evidence in support of the above conjectures.

\begin{thm} \label{transferzltame}
  Let $\xi$ be as in Conjecture \ref{transferzl}, and suppose $\xi$ is tame. Then there is
  an equivalence of categories $\Rep_{\phi',\oZl}(G')\simto
  \Rep_{\phi,\oZl}(G)$. Equivalently, let $\phi$ be as in Conjecture \ref{reduc_unipzl} and suppose
  $\phi$ is tame. Then there is an equivalence of categories $\Rep_{1,\oZl}(G_{\phi})\simto \Rep_{\phi,\oZl}(G)$. 
\end{thm}

  Beyond the restriction to tame parameters, what this theorem is
  missing at the moment is the compatibility with the 
  transfer of $\oQl$-irreducible representations. 
This theorem is not proved in this paper. We will only show in \ref{prooftransferzl} how it follows from the
results in \cite{Datequiv} where we construct equivalences in the specific  cases where  $\xi$ is either an
unramified automorphic induction, or a totally ramified base
change. Let us also note that these cases are not
obtained via Hecke algebra techniques, but by importing results from
Deligne-Lusztig theory via coefficient systems on buildings.

Remark : a less precise version of the second half of the theorem 
is that any level $0$ $\oZl$-block of $\GL_{n}$ is
equivalent to the principal $\oZl$-block of an unramified group of
type $\GL$.

\alin{A possible reduction to tame parameters}
Here we reinterpret current work in progress by G. Chinello in our
setting and show how it will imply (if successful) that Theorem
\ref{tamecase} remains true without the word ``tame''.
For this, we push our formalism so as to
reduce the general case to the tame case in the following way.
Instead of considering parameters with source $I_{F}$ or $I_{F}^{(\ell)}$,
consider the set 
$$\Phi_{\rm wild}(\mathbf{G},\oQl):=\left\{\psi:P_{F}\To{}{^{L}\mathbf{G}(\oQl)},\exists \varphi\in
\Phi(\mathbf{G},\oQl), \varphi_{|P_{F}}=\psi\right\}.$$
To any $\psi$ as above is attached a direct factor (no longer a block) $\Rep_{\psi}(G):= \bigoplus_{\phi_{|P_{F}}=\psi}\Rep_{\phi}(G)$ of
$\Rep_{\oZl}(G)$. When $\psi$ is trivial, $\Rep_{\psi}(G)$ is nothing but the level $0$
subcategory of $\Rep_{\oZl}(G)$.

The same procedure as in \ref{reduc_unip} provides us with a 
factorization $
\psi:\,\, P_{F} \To{1}
{^{L}\mathbf{G}_{\psi}}  \To{\xi}
{^{L}\mathbf{G}} 
$
 where  $\mathbf{G}_{\psi}$ is a
 tamely ramified group of $\GL$-type over $F$ such that $\hat\mathbf{G}_{\psi}=C_{\hat\mathbf{G}}(\psi)$.
In this setting, the map $\phi'\mapsto \xi\circ\phi'$ is a bijection 
$$\{\phi'\in \Phi_{\ell'-\rm inert}(\mathbf{G}_{\psi},\oQl) \hbox{ tame}\}  
\simto
\{\phi\in \Phi_{\ell'-\rm inert}(\mathbf{G},\oQl), \phi_{|P_{F}}=\psi\}$$
and moreover $\xi$ induces an isomorphism 
$C_{\hat\mathbf{G}_{\psi}}(\phi')\simto
C_{\hat\mathbf{G}}(\xi\circ\phi')$. Therefore, Conjecture \ref{reduc_unipzl} implies the
following one :

\begin{con}
  There is an equivalence of categories $\Rep_{1}({G}_{\psi})\simto\Rep_{\psi}(G)$ that
  extends the transfer associated to the embedding $\xi:\,^{L}\mathbf{G}_{\psi}\injo {^{L}\mathbf{G}}$. 
\end{con}

Conversely, if the equivalence predicted in the last  conjecture exists, it
 has to restrict to an equivalence
$\Rep_{\phi'}(G_{\psi})\simto \Rep_{\xi\circ\phi'}(G)$ for all 
$\phi'\in \Phi_{\ell'-\rm inert}(\mathbf{G}_{\psi},\oQl)$. Therefore the latter
conjecture, together with Conjecture \ref{reduc_unipzl} restricted to tame parameters,
implies the full Conjecture \ref{reduc_unipzl}. The same is true if we weaken all these
statements by removing the compatibility with transfer of irreducible $\oQl$-representations.


Now, the point is that the weakened form of the last conjecture can be attacked by Hecke algebra
techniques. Namely, the core of the problem is to exhibit an isomorphism between the Hecke
algebra of a simple character (or rather, of its $\beta$-extension) of $\GL_{n}(F)$ and that
of the first principal congruence subgroup of an appropriate $\GL_{n'}(F')$. This is exactly
what Chinello is currently doing.

\subsection{More general groups}

Because its formulation fits well with Langlands' functoriality
principle, a suitable version of Theorem \ref{transferql}
should hold for all $L$-homomorphisms. In this subsection we speculate on how it should
work in an ``ideal world'', in which 
Langland's parametrization is known and satisfies some natural properties.
In a forthcoming work, we will treat groups of \emph{classical type,} meaning groups which
are products of restriction of scalars of quasi-split classical groups, 
where all we need is available, and the desired equivalence of categories will be
extracted from the work of Heiermann \cite{Heiermann_equiv}.

\alin{An ideal world} \label{ideal}
Suppose we knew the existence of a coarse Langlands' parametrization
map $\Irr{\oQl}{G}\To{}  \Phi(\mathbf{G},\oQl)$, $\pi\mapsto \varphi_{\pi}$,
 for any reductive group $\mathbf{G}$ over  any $p$-adic field $F$,
 and suppose further that these parametrizations were compatible with
 parabolic induction as in \cite[Conj. 5.2.2]{Hainesstable}.
This means that if $\pi$ is an irreducible subquotient of some
parabolically induced
representation $\ip{M,P}{G}\sigma$, 
then ${\varphi_{\pi}}_{|W_{F}} \sim {^{L}\iota}\circ {\varphi_{\sigma}}_{|W_{F}}$, where $^{L}\iota :
{^{L}\mathbf{M}} \injo {^{L}\mathbf{G}}$ is any embedding dual to $\mathbf{M}\injo \mathbf{G}$.
Then, as in the case of groups of $\GL$-type, Bernstein's decomposition implies a
decomposition
$$ \Rep_{\oQl}(G) =\prod_{\phi\in \Phi_{\rm inert}(\mathbf{G},\oQl)} \Rep_{\phi}(G)$$
where $\Phi_{\rm inert}(\mathbf{G},\oQl)\subset H^{1}(I_{F},\hat\mathbf{G}(\oQl))^{W_{F}}$
 is the set of $\hat\mathbf{G}$-conjugacy classes of
continuous sections $I_{F}\To{}{^{L}\mathbf{G}}(\oQl)$ that admit an extension to an $L$-parameter in
$\Phi(\mathbf{G},\oQl)$, and the direct factor category 
$\Rep_{\phi}(G)$ is characterized by its simple objects, which are  
 the irreducible
representations $\pi$ such ${\varphi_{\pi}}_{|I_{F}}\sim \phi$. 

We note that these desiderata are now settled for quasi-split classical groups. Namely,
the existence of Langlands' parametrization was obtained by Arthur for symplectic and
orthogonal groups and by Mok for unitary groups, while the compatibility with parabolic
induction follows from work of Moeglin for all these groups.

The main difference with
the case of groups of $\GL$-type is that $\Rep_{\phi}(G)$ is not necessarily a single Bernstein
block. For example $\Rep_{1}({\rm Sp}_{4}(F))$ contains the principal series
block and a supercuspidal unipotent representation.
 Equivalently, the corresponding idempotent $e_{\phi}$ of $\ZG_{\oQl}(G)$ need not be primitive.
Note that $e_{\phi}$ actually lies in the ``stable'' center $\ZG^{\rm st}_{\oQl}(G)$  defined in
\cite[5.5.2]{Hainesstable}, since an $L$-packet is either contained in $\Rep_{\phi}(G)$ or disjoint
from it. However, the following example shows that $e_{\phi}$ needs not even be primitive in this stable center.

\medskip
\emph{Example.} Suppose $\mathbf{G}=\SL_{2}$ with $p$ odd, and let
$\phi$ be given by $i\in I_{F}\mapsto {\rm diag}(\varepsilon(i),1)\in \PGL_{2}$ with $\varepsilon$
the unique non-trivial quadratic character of $I_{F}$. Then an extension $\varphi$ of
$\phi$ to $W_{F}$ has two possible shapes : either it is valued in the maximal torus of
$\PGL_{2}$ or it takes any Frobenius substitution to an order $2$ element that normalizes
non trivially this torus. In the language of \cite[Def. 5.3.3]{Hainesstable}, these
extensions (called infinitesimal characters in \emph{loc. cit.}) fall in two distinct
inertial classes $[\varphi_{ps}]\sqcup[\varphi_{0}]$. Accordingly, we have a decomposition
$\Rep_{\phi}(G)=\Rep_{[\varphi_{ps}]}(G)\times \Rep_{[\varphi_{0}]}(G)$, where
$\Rep_{[\varphi_{ps}]}(G)$ is the block formed by (ramified) principal series associated
to the character $\varepsilon\circ{\rm Art}_{F}^{-1}$ of the maximal compact subgroup of
the diagonal torus of $\SL_{2}(F)$, while $\Rep_{[\varphi_{0}]}(G)$ is the category
generated by cuspidal representations in the $L$-packet associated to $\varphi_{0}$. The
cardinality of this $L$-packet is that of the centralizer of $\varphi_{0}$, i.e. $4$, so
 that $\Rep_{[\varphi_{0}]}(G)=\Rep_{\oQl}(\{1\})^{\times 4}$. Accordingly, the
 idempotent $e_{\phi}\in\ZG_{\oQl}(G)$ decomposes as
 $e_{\phi}=e_{[\varphi_{ps}]}+e_{[\varphi_{0}]}$ in $\ZG_{\oQl}(G)$, with both
 $e_{[\varphi_{ps}]}$, $e_{[\varphi_{0}]}$ belonging to the ``stable Bernstein center''
 (as in  \cite[5.5.2]{Hainesstable}), showing that $e_{\phi}$ is not primitive, even in the ``stable sense''.

 \begin{rak}
   The decomposition of $\Rep_{\phi}(G)$ in the above example can be
   generalized whenever the centralizer $C_{\hat\mathbf{G}}(\phi)$ is
   \emph{not connected}. To see how, let us choose an extension
   $\varphi$ of $\phi$ to $W_{F}$, and let us take up the procedure of
   \ref{reduc_unip}. So, conjugacy under $\varphi(w)$ still endows
   $C_{\hat\mathbf{G}}(\phi)$, and also
   $C_{\hat\mathbf{G}}(\phi)^{\circ}$, with an action of
   $W_{F}/I_{F}$. But while the outer action $W_{F}/I_{F}\To{}{\rm
     Out}(C_{\hat\mathbf{G}}(\phi))$ is still independent of
   $\varphi$, the outer action $W_{F}/I_{F}\To{}{\rm
     Out}(C_{\hat\mathbf{G}}(\phi)^{\circ})$ actually depends on
   $\varphi$. More precisely, if $\eta_{\varphi}$ denotes the image of
   $\rm Frob$ by this outer action, then the set $A_{\phi}$ of all
   possible $\eta_{\varphi}$ is a single $\pi_{0}(C_{\hat\mathbf
     G}(\phi))$-orbit inside ${\rm
     Out}(C_{\hat\mathbf{G}}(\phi)^{\circ})$.
   Now, observe that if $\varphi, \varphi'$ are inertially equivalent
   in the sense of \cite[Def. 5.3.3]{Hainesstable}, then
   $\eta_{\varphi}=\eta_{\varphi'}$. This is because $\varphi'({\rm
     Frob})=z\varphi({\rm Frob})$ for some $z$ that belongs to some
   torus contained in $C_{\hat\mathbf{G}}(\phi)$.  Therefore we get a
   decomposition
$$ \Rep_{\phi}(G)=\prod_{\eta\in A_{\phi}} \Rep_{\phi,\eta}(G)$$
where $\Rep_{\phi,\eta}(G)$ is the ``stable'' Bernstein summand of
$\Rep_{\oQl}(G)$ whose irreducible objects $\pi$ satisfy
${\varphi_{\pi}}_{|I_{F}}\sim\phi$ and $\eta_{\varphi}=\eta$.
It is plausible that the corresponding idempotents are primitive in the stable Bernstein center.
\end{rak}

\alin{The transfer problem}
 Suppose given an $L$-morphism $\xi:\,{^{L}\mathbf{G'}}\To{}{^{L}\mathbf{G}}$ and an
inertial parameter $\phi'\in\Phi_{\rm inert}(\mathbf{G'},\oQl)$ such that $\xi$ induces an isomorphism
$C_{\hat\mathbf{G}'}(\phi') \simto C_{\hat\mathbf{G}}(\phi)$. In this generality, 
new issues  arise on the path to a possible generalization of Theorem \ref{transferql}.

The first one  is related to the internal structure of $L$-packets. Suppose temporarily  that
$\mathbf{G}$ and $\mathbf{G'}$ are \emph{quasi-split}. It is then expected that
the $L$-packet of $\Irr{\oQl}{G'}$ associated to an extension $\varphi'$ of $\phi$  is parametrized by 
irreducible representations of the component group $\pi_{0}(C_{\hat
  \mathbf{G'}}(\varphi')/Z(\hat\mathbf{G}')^{W_{F}})$. In our
setting, $\xi$ has to induce an isomorphism 
$C_{\hat \mathbf{G'}}(\varphi')\simto C_{\hat\mathbf{G}}(\xi\circ\varphi')$.
\begin{lem}
If
$\ker(\xi)$ is commutative\footnote{With a bit more work, we can
  weaken the hypothesis to :   $Z((\ker(\xi)^{\circ})_{\rm der})$ has
  order prime to $p$. The lemma may be true with no hypothesis at
  all}, $\xi$ also induces an isomorphism
$Z(\hat\mathbf{G}')^{W_{F}}\simto C_{\hat\mathbf{G}}(\xi)$.
\end{lem}
\begin{proof}
  Here, as usual, $C_{\hat\mathbf{G}}(\xi)$ is the centralizer of the image of $\xi$. So
  we clearly have $\xi(Z(\hat\mathbf{G}')^{W_{F}})\subset
  C_{\hat\mathbf{G}}(\xi)$. Moreover, since $Z(\hat\mathbf{G}')^{W_{F}}\subset C_{\hat
    \mathbf{G'}}(\phi')$, our running assumptions imply that
  $\xi_{|Z(\hat\mathbf{G}')^{W_{F}}}$ is injective. It remains to prove
  surjectivity. Again we have $C_{\hat\mathbf{G}}(\xi)\subset C_{\hat\mathbf{G}}(\phi)$,
  so any element of $C_{\hat\mathbf{G}}(\xi)$ has the form $\xi(x)$ for a unique $x\in
  C_{\hat\mathbf{G}'}(\phi')$, and we need to prove that $x\in
  Z(\hat\mathbf{G}')^{W_{F}}$.

Pick an extention $\varphi'$ of $\phi'$. Since $\im(\varphi')$ normalizes $\im(\phi')$, it
also normalizes $C_{\hat\mathbf{G}'}(\phi')$, so that $[x,\im(\varphi')]\subset
C_{\hat\mathbf{G}'}(\phi')\cap \ker(\xi)=\{1\}$. On the other hand,
$[x, \hat\mathbf{G}']= [x,\hat\mathbf{G}'_{\rm der}]\subset
\hat\mathbf{G}'_{\rm der}\cap \ker(\xi)$ which is finite. Since
 $\hat\mathbf{G}'$  is connected, it follows that
$[x,\hat\mathbf{G}']=\{1\}$. Finally, since
$^{L}\mathbf{G'}=\im(\varphi')\hat\mathbf{G}'$, we get $[x,{^{L}\mathbf{G}'}]=\{1\}$,
\emph{i.e.} $x\in Z(\hat\mathbf{G}')^{W_{F}}$.
\end{proof}
Assume from now on that $\ker(\xi)$ is commutative.
Since $Z(\hat\mathbf{G})^{W_{F}}\subset C_{\hat\mathbf{G}}(\xi)$, we get an embedding of
$Z(\hat\mathbf{G})^{W_{F}}$ in $Z(\hat\mathbf{G}')^{W_{F}}$,
 whence a map
$$h_{\xi}:\,H^{1}(F,\mathbf{G'})\To{}H^{1}(F,\mathbf{G}),$$ 
through Kottwitz's isomorphism \cite[(6.4.1)]{KottCusp}. 
Recall now that to any $\alpha\in
H^{1}(F,\mathbf{G})$ is associated a ``pure'' inner form $\mathbf{G}_{\alpha}$ of
$\mathbf{G'}$. 


\begin{expec} With the foregoing assumptions, for any $\alpha\in H^{1}(F,\mathbf{G})$ 
  there should exist an equivalence of categories $\xi_{*}:\,\prod_{\beta\in h_{\xi}^{-1}(\alpha)}
  \Rep_{\phi'}(G'_{\beta})\simto \Rep_{\phi}(G_{\alpha})$ such that, for any
  irreducible $\pi'\in \Rep_{\phi'}(G'_{\beta})$ we have $\varphi_{\xi_{*}(\pi')}=\xi\circ\varphi_{\pi'}$
\end{expec}

\emph{Example.} Suppose $\mathbf{G}=\SL_{2}$, $\mathbf{G'}=\mathbf{U}(1)$ (norm $1$
elements in a quadratic unramified extension) and $\xi:\, \GM_{m}\rtimes
W_{F}\To{}\PGL_{2}\times W_{F}$ is the automorphic induction homomorphism. Then start with
$\phi'=\theta:\,I_{F}\To{}\GM_{m}$ a character such that $\theta^{\sigma^{2}}=\theta$ and
$\theta^{\sigma}\theta^{-1}$ has order $>2$ (with $\sigma$ a Frobenius element). Then $\Rep_{\phi}(G)$ is generated by $2$
distinct irreducible cuspidal representations of $G$. The centralizer
$C_{\hat\mathbf{G}'}(\phi)=\hat\mathbf{G}'$ is connected,  we have $H^{1}(F,\mathbf{G})=\{1\}$
while $H^{1}(F,\mathbf{G}')$ has $2$ elements. Both pure inner forms of $\mathbf{G'}$ are isomorphic to
$\mathbf{G'}$ and $\Rep_{\phi'}(G')$ is generated by a single irreducible
representation (a character). This picture generalizes to supercuspidal $L$-packets
constructed in \cite{dBR}. 

\medskip

\def\SO{{\rm SO}}
\def\Sp{{\rm Sp}}
\def\SL{{\rm SL}}
\emph{Example.} Suppose that $p$ is odd, $\mathbf{G}=\SO_{5}$, $\mathbf{G'}=\SO_{3}\times\SO_{3}$ 
 and that $\xi:\, \SL_{2}\times\SL_{2}\injo \Sp_{4}$  is an isomorphism onto the
 centralizer of the element $\xi(1,-1)$ in $\Sp_{4}$. Take $\phi':=(1,\varepsilon)$ with
 $\varepsilon$ the unique quadratic non trivial character of $I_{F}$. In particular,
 $C_{\hat\mathbf{G}'}(\phi')=\SL_{2}\times\SL_{2}\simto C_{\hat\mathbf{G}}(\phi)$ is connected.
We have $H^{1}(F,\mathbf{G})=\{\pm 1\}$,  $H^{1}(F,\mathbf{G}')=\{\pm 1\}\times\{\pm 1\}$
and $h_{\xi}$ is the multiplication map. The category $\Rep_{\phi'}(G')$ is a Bernstein block coming
from the maximal torus. On the other hand, ${G'}_{(-1,-1)} =
(D^{\times}/F^{\times})\times(D^{\times}/F^{\times})$, with $D$ the quaternion algebra, so that 
$\Rep_{\phi'}(G'_{(-1,-1)})$ decomposes into 4 blocks, each one  generated by a quadratic character of the form
$\psi\cdot\chi$ with $\psi$ a quadratic unramified character of $D^{\times}$ and $\chi$ a
quadratic ramified character of $D^{\times}$.
Accordingly, $\Rep_{\phi}(G)$ is the sum of a Bernstein block coming from the torus and
$4$ supercuspidal blocks, associated to the four Langlands parameters $\varphi=\xi\circ
\varphi'$ with $\varphi' = (\psi,\chi)\otimes \Delta : W_{F}\times\SL_{2}\injo
\SL_{2}\times\SL_{2}$ and where now $\psi$ is a quadratic unramified character of $W_{F}$ and $\chi$ a
quadratic ramified character of $W_{F}$.

\medskip

A second issue towards a generalization of Theorem \ref{transferql}
arises when we try to go to more general non quasi-split groups. A similar ``expectation'',
involving Kottwitz' $B(G)_{\rm bas}$ instead of $H^{1}(F,G)$, might apply to  ``extended pure forms'' of
quasi-split groups, enabling one to reach inner forms of groups with connected center.

Another possibility is to add a suitable relevance condition. For example, consider condition
\begin{center}
(R)   An extension $\varphi'\in H^{1}(W_{F},\hat\mathbf{G})$ of $\phi'$ is relevant
  for $\mathbf{G'}$ if and only if $\xi\circ\varphi'$ is relevant for $\mathbf{G}$. 
\end{center}
The following statement seems to pass the crash-test of inner forms of linear groups :

\medskip
\emph{Let $\xi$ be as in the begining of this paragraph, and suppose further that
  condition (R) is satisfied.
Then  there is an equivalence of categories $\xi_{*}:\,\prod_{\beta\in \ker(h_{\xi}^{-1})}
  \Rep_{\phi'}(G'_{\beta})\simto \Rep_{\phi}(G)$ such that, for any
  irreducible $\pi'\in \Rep_{\phi'}(G'_{\beta})$ we have $\varphi_{\xi_{*}(\pi')}=\xi\circ\varphi_{\pi'}$.
}

\alin{The ``reduction to unipotent'' problem}
Here the obvious new difficulty is that the centralizer $C_{\hat\mathbf{G}}(\phi)$ may not
be connected. When it is connected, the same procedure as for groups of $\GL$-type allows us
to associate to $\phi$ an unramified group
$\mathbf{G}_{\phi}$ with $\hat\mathbf{G}_{\phi}=C_{\hat\mathbf{G}}(\phi)$, together with 
factorization(s) of $\phi$ as 
$I_{F}\To{1} {^{L}\mathbf{G}}_{\phi}\To{\xi} {^{L}\mathbf{G}}$, see \ref{Gphiconnected}. Therefore, the natural expectation is :

\begin{expec} Assume $\mathbf{G}$  quasi-split, $C_{\hat\mathbf{G}}(\phi)$ 
  connected,  and let $\xi$ be as above. Then
for any $\alpha\in H^{1}(F,\mathbf{G})$ 
  there should exist an equivalence of categories $\xi_{*}:\,\prod_{\beta\in h_{\xi}^{-1}(\alpha)}
  \Rep_{1}(G_{\phi,\beta})\simto \Rep_{\phi}(G_{\alpha})$ such that, for any
  irreducible $\pi'\in \Rep_{\phi'}(G'_{\beta})$ we have
  $\varphi_{\xi_{*}(\pi')}=\xi\circ\varphi_{\pi'}$.
\end{expec}

\emph{Examples.} In the previous example with $\mathbf{G}=\SL_{2}$, we have
$\mathbf{G}_{\phi}=\mathbf{U}(1)$ and the expectation is therefore satisfied. More
generally, for $\phi$ the restriction of a tame parameter corresponding to a supercuspidal
$L$-packet as considered in \cite{dBR}, the expectation holds (note that in this case,
$\mathbf{G}_{\phi}$ is an anisotropic unramified torus). Also the last example above gives
us an instance of this expectation in which
$\mathbf{G}=\SO_{5}$ and $\mathbf{G}_{\phi}=\SO_{3}\times\SO_{3}$.

\medskip
More mysterious is the case when $C_{\hat\mathbf{G}}(\phi)$ is not connected. In
\ref{Gphinonconnected} we 
define several non-connected reductive groups $\mathbf{G}_{\phi}^{\tau}$ over $F$, where
$\tau$ belongs to a 
set $\o\Sigma(\phi)$ equipped with a map to $H^{1}(F,\mathbf{G})$. 
We think that a similar
statement to that above is plausible, with this collection of groups replacing the
$\mathbf{G}_{\phi,\beta}$, at least when $C_{\hat\mathbf{G}}(\phi)$ is ``quasi-split'' in
the sense that it is isomorphic to
$C_{\hat\mathbf{G}}(\phi)^{\circ}\rtimes\pi_{0}$ with $\pi_{0}$ acting by some
\'epinglage-preserving automorphisms. 

More precisely, for these non-connected groups
 $\mathbf{G}_{\phi}^{\tau}$, there is a natural notion of ``unipotent factor''
 $\Rep_{1}(G_{\phi}^{\tau})$ of the
 category $\Rep_{\oQl}(G_{\phi}^{\tau})$. Namely a representation of
 $G_{\phi}^{\tau}$ is unipotent if the restriction to
 $\mathbf{G}_{\phi}^{\tau,\circ}(F)$ is unipotent. 
Now,  when $C_{\hat\mathbf{G}}(\phi)$ is ``quasi-split'', we expect that $\Rep_{\phi}(G)$
will be equivalent to the product of all $\Rep_{1}(G_{\phi}^{\tau})$ with $\tau$
mapping to $1\in H^{1}(F,\mathbf{G})$.

\medskip
\emph{Example.}  Take up the example of paragraph \ref{ideal} for
$\mathbf{G}=\SL_{2}$. For the $\phi$ considered there, $C_{\hat\mathbf{G}}(\phi)$ is the
normalizer of the diagonal torus. Our construction, explicitly detailed in \ref{examplesl_2}, provides 3 groups :
$\mathbf{G}_{\phi}^{ps}=\GM_{m}\rtimes\ZM/2\ZM$ and  $\mathbf{G}_{\phi}^{1}=\mathbf{G}_{\phi}^{2}=\mathbf{U(1)}\rtimes \ZM/2\ZM$, with the
generator of $\ZM/2\ZM$ acting by the inverse map in each case.
For $i=1,2$, the unipotent factor of $\Rep_{\oQl}(G_{\phi}^{i})$ consists of representations
that are trivial on $U(1)$, therefore $\Rep_{1}(G_{\phi}^{i})=
\Rep_{\oQl}(\ZM/2\ZM)\simeq\Rep_{\oQl}(\{1\})^{2}$ and these two copies
account for the supercuspidal factor $\Rep_{[\varphi_{0}]}(G)=\Rep_{\oQl}(\{1\})^{4}$. On the
other hand, $\Rep_{1}(G_{\phi}^{ps})=\Rep_{\oQl}(\ZM\rtimes\ZM/2\ZM)$ which is
indeed equivalent to $\Rep_{[\varphi_{ps}]}(G)$.

\medskip
We hope that the recent results of Heiermann \cite{Heiermann_equiv} will enable us to confirm the
above ``expectations'' for groups of classical type. In this case, the disconnected
centralizers are explained by even orthogonal factors (in the last example, the groups $\mathbf{G}_{\phi}$
are ``pure'' inner forms of $\mathbf{O}_{2}$).

\section{Details and proofs}

\emph{Notation.} Unless stated otherwise,  $\hat\mathbf{G}$ and
$^{L}\mathbf{G}$ will stand respectively for $\hat\mathbf{G}(\oQl)$ and
$^{L}\mathbf{G}(\oQl)$.

Given an exact sequence $H\injo \wt H\twoheadrightarrow W$ of topological groups,
we denote by $\Sigma(W,\wt H)$ the set of continuous group sections $W\To{} \wt H$ that split the
sequence and by $\o\Sigma(W,\wt H)$ the set of $H$-conjugacy classes in $\Sigma(W,\wt H)$.
If we fix $\sigma\in\Sigma(W,\wt H)$, conjugation by $\sigma(w)$ induces an action $\alpha_{\sigma}$ of $W$
on $H$ and a set-theoretic continuous projection $\pi_{\sigma}: \wt H\To{} H$. Then the map
$\sigma'\mapsto \pi_{\sigma}\circ \sigma'$ is a bijection $\Sigma(W,\wt H)\simto
Z^{1}_{\alpha_{\sigma}}(W,H)$ that descends to a bijection $\o\Sigma(W,\wt H)\simto
H^{1}_{\alpha_{\sigma}}(W,H)$. 

Suppose that $H=\mathbf{H}(\oQl)$ for some algebraic group $\mathbf{H}$.
We will say that the extension $\wt H$ of $W$ by $H$ is \emph{almost algebraic} if it is
the pullback of an extension of some \emph{finite} quotient of $W$ by $H$.
Equivalently, some
finite index subgroup of $W$ lifts to a normal subgroup $W' \subset \wt H$ that commutes with $H$. 
In this case, a
section $\sigma\in\Sigma(W,\wt H)$ is called \emph{admissible} if for any such $W'$
(equivalently, some $W'$) the elements $\sigma(w)$, $w\in W$  are \emph{semi-simple} in the quotient $\wt H/W'$
(which is the group of $\oQl$-points of an algebraic group).

For example, $^{L}\mathbf{G}$ is an almost algebraic extension of $W_{F}$ by
$\hat\mathbf{G}$, and the set $\Phi_{\rm Weil}(\mathbf{G},\oQl)$ of
\emph{admissible Weil parameters} (not Weil-Deligne)
is the set of admissible elements in $\Sigma(W_{F},{^{L}\mathbf{G}})$.

\subsection{The centralizer and its dual $F$-groups}
We start with a general reductive group $\mathbf{G}$ over $F$. We will
denote by  $K_{F}$ a closed normal subgroup of $W_{F}$ contained in $I_{F}$,
and we fix $\phi : K_{F}\To{}{^{L}\mathbf{G}}$  an \emph{admissible}
$K_{F}$-parameter, \emph{i.e.} the restriction to $K_{F}$ of an
admissible Langlands parameter.
As previously introduced, we denote by $C_{\hat\mathbf{G}}(\phi)$ the centralizer of
$\phi(K_{F})$ in $\hat\mathbf{G}$. By \cite[Lemma 10.1.1]{KottCusp} this is a  reductive, possibly non-connected,
subgroup of $\hat\mathbf{G}$.

\alin{Extensions of $\phi$ to $W_{F}$} \label{extensions}
By hypothesis, $\phi$ can be extended to an admissible Weil parameter $\varphi :
W_{F}\To{}{^{L}\mathbf{G}}$. Because $\varphi(w)$ normalizes $\phi(K_{F})$, it also
normalizes $C_{\hat\mathbf{G}}(\phi)$ so that, letting $w$ act by  conjugation under $\varphi(w)$, we get an action 
$$\alpha_{\varphi}:\, W_{F}/K_{F} \To{}
\Aut(C_{\hat\mathbf{G}}(\phi)).$$
Note that the restriction of this action to a finite index subgroup of $W_{F}$ is by inner
automorphisms of $C_{\hat\mathbf{G}}(\phi)$. Indeed, if $F'$ splits $\mathbf{G}$, the  
action of $W_{F'}$ is through conjugation inside the normalizer
$\NC=\NC_{\hat\mathbf{G}}(C_{\hat\mathbf{G}}(\phi))$, but $\NC^{\circ}$ has finite
index in $\NC$ and acts by inner
automorphisms since ${\rm Out}(C_{\hat\mathbf{G}}(\phi))$ is discrete.

 Now, if we pick another extension $\varphi'$ and write $\varphi'(w)=
\eta(w)\varphi(w)$ with $\eta(w)\in\hat\mathbf{G}$, then we compute that  $\eta\in
Z^{1}_{\alpha_{\varphi}}(W_{F}, C_{\hat\mathbf{G}}(\phi))$ (a 
$1$-cocycle for the action $\alpha_{\varphi}$).
So we see in particular that 
\begin{enumerate}
\item the outer action $W_{F} \To{} {\rm Out}(C_{\hat\mathbf{G}}(\phi))$ is independent of
  $\varphi$ and factors over a finite quotient.
\item the subgroup $\wt C_{\hat\mathbf{G}}(\phi):=C_{\hat\mathbf{G}}(\phi).\varphi(W_{F})$ of $^{L}\mathbf{G}$ is independent of $\varphi$.
 \item the action of $W_{F}$ on the center $Z(C_{\hat\mathbf{G}}(\phi))$ via $\alpha_{\varphi}$ is
   independent of $\varphi$ and factors over a finite quotient.
\end{enumerate}
Let us denote by $Z^{1}(W_{F},\hat\mathbf{G})_{\phi}$ the set of all cocycles that extend
$\phi$ and by $H^{1}(W_{F},\hat\mathbf{G})_{[\phi]}$ the fiber of the equivalence class of
$\phi$ in $H^{1}(K_{F},\hat\mathbf{G})$. Then the map $\eta\mapsto \eta\varphi$ clearly is
a bijection
$$Z^{1}_{\alpha_{\varphi}}(W_{F}/K_{F},C_{\hat\mathbf{G}}(\phi)) \simto
Z^{1}(W_{F},\hat\mathbf{G})_{\phi}$$
and it is easily checked that it induces a bijection
$$H^{1}_{\alpha_{\varphi}}(W_{F}/K_{F},C_{\hat\mathbf{G}}(\phi)) \simto
H^{1}(W_{F},\hat\mathbf{G})_{[\phi]}.$$
In order to see how admissibility is carried through this bijection, let us recast it in
terms of sections. Consider the extension
$$ C_{\hat\mathbf{G}}(\phi)\injo \wt C_{\hat\mathbf{G}}(\phi)\twoheadrightarrow W_{F}$$
where the middle group is that of point ii) above and the map to $W_{F}$ is induced by
the projection $^{L}\mathbf{G}\To{} W_{F}$. Quotienting by  $\phi(K_{F})$ we get an extension
$$ C_{\hat\mathbf{G}}(\phi)\injo \wt C_{\hat\mathbf{G}}(\phi)/\phi(K_{F})\twoheadrightarrow W_{F}/K_{F}.$$
Then we have 
$$\Sigma(W_{F},{^{L}\mathbf{G}})_{\phi}=\Sigma(W_{F},\wt C_{\hat\mathbf{G}}(\phi))_{\phi}
\simto \Sigma(W_{F}/K_{F},\wt C_{\hat\mathbf{G}}(\phi))$$
where the index $\phi$ means ``extends $\phi$'' and the last map takes a section
$\varphi$ to ($\varphi$ mod $\phi(K_{F})$).
The next lemma shows that $\wt C_{\hat\mathbf{G}}(\phi)$, and hence $\wt
C_{\hat\mathbf{G}}(\phi)/\phi(K_{F})$, is ``almost algebraic'' in the sense of the begining
of this section, and also that admissibility is preserved through the last bijection,
giving a bijection
$$ \Phi_{\rm Weil}(\mathbf{G},\oQl)_{[\phi]}\simto
\o\Sigma(W_{F}/K_{F},\wt C_{\hat\mathbf{G}}(\phi)/\phi(K_{F}))_{\rm adm}.$$ 

\begin{lem}
  There exist a finite extension $F'$ of $F$ and an extension $\varphi':W_{F}\To{}{^{L}\mathbf{G}}$ of $\phi$  such that  $\varphi'(w')= (1,w')$ for all $w'\in W_{F'}$. 
\end{lem}
\begin{proof}
  Start with any extension $\varphi$ as above. 
Let $F_{0}$ be a finite Galois extension of $F$
  that splits $\mathbf{G}$ and denote by $\bar\varphi$ the composition of $\varphi$ with
  the projection ${^{L}}\mathbf{G}\To{}\hat\mathbf{G}\rtimes \Gamma_{F_{0}/F}$. 
 Then $\bar\varphi(W_{F})$ is an extension of a cyclic (possibly infinite) group by a finite group and 
there is some finite extension $F'$ of $F$ such that $\bar\varphi(W_{F'})$ is central
 in $\bar\varphi(W_{F})$  and $\bar\varphi(I_{F'})$ is trivial. We may also assume that
 $F'$ contains $F_{0}$, so that $\bar\varphi(W_{F'})\subset \hat\mathbf{G}$ and actually 
$\bar\varphi(W_{F'})\subset C_{\hat\mathbf{G}}(\bar\varphi)=C_{\hat\mathbf{G}}(\varphi)$. Enlarging $F'$ further, we may
assume that $\bar\varphi(W_{F'})\subset C_{\hat\mathbf{G}}(\varphi)^{\circ}$. Now, since
$W_{F}/I_{F}$ is cyclic and contains $W_{F'}/I_{F'}$ with finite index, we can find a homomorphism
$\chi:\,W_{F}/I_{F}\To{} C_{\hat\mathbf{G}}(\varphi)^{\circ}$ such that
$\chi_{|W_{F'}}=\bar\varphi_{|W_{F'}}$. Then $\varphi':=\chi^{-1}\varphi$ has the desired property.
\end{proof}

\alin{The group $\mathbf{G}_{\phi}$  in the connected case} \label{Gphiconnected}
Let us assume in this paragraph that $C_{\hat\mathbf{G}}(\phi)$ is \emph{connected}, and
let $\mathbf{G}_{\phi}^{\rm split}$ denote a dual group for $C_{\hat\mathbf{G}}(\phi)$
defined over $F$.  By item i) of \ref{extensions}, we have a canonical outer action
$W_{F}/K_{F}\To{} {\rm Out}(C_{\hat\mathbf{G}}(\phi))={\rm
  Out}(\mathbf{G}_{\phi}^{\rm split})$ that factors over a finite
quotient. Choosing a section 
${\rm  Out}(\mathbf{G}_{\phi}^{\rm split}) \To{}
{\rm  Aut}(\mathbf{G}_{\phi}^{\rm split})$ (i.e. choosing an \'epinglage of
$\mathbf{G}_{\phi}^{\rm split}$), we get an action of the Galois group $\Gamma_{F}$ on
$\mathbf{G}_{\phi}^{\rm split}$. Associated to this action is an $F$-form, 
$\mathbf{G}_{\phi}$ of $\mathbf{G}_{\phi}^{\rm split}$, which by construction is
 a quasi-split connected reductive group
 over $F$, uniquely defined  up to $F$-isomorphism, and which splits over an extension
 $F'$ such that $K_{F'}=K_{F}$.

For any continous cocycle $\tau: \Gamma_{F}\To{}\mathbf{G}_{\phi}(\o F)$, we have an inner
form $\mathbf{G}_{\phi}^{\tau}$ of $\mathbf{G}_{\phi}$ over $F$. If $\tau'$ is
cohomologous to $\tau$, there is an $F$-isomorphism $\mathbf{G}_{\phi}^{\tau'}\simto
\mathbf{G}_{\phi}^{\tau}$ well-defined up to inner automorphism. We will therefore
identify $\mathbf{G}_{\phi}^{\tau}$ and $ \mathbf{G}_{\phi}^{\tau'}$, and we now have a
collection $(\mathbf{G}_{\phi}^{\tau})_{\tau\in H^{1}(F,\mathbf{G}_{\phi})}$ of
$F$-groups associated to $\phi$.

In order to shrink this collection, 
we now define a map $H^{1}(F,\mathbf{G}_{\phi})\To{h_{\phi}}H^{1}(F,\mathbf{G})$ by using the
Kottwitz isomorphism. Recall the latter is an isomorphism 
$H^{1}(F,\mathbf{G})\simto \pi_{0}(Z(\hat\mathbf{G})^{W_{F}})^{*}$ with $*$ denoting a
Pontrjagin dual.
In the case of $\mathbf{G}_{\phi}$, this reads 
$H^{1}(F,\mathbf{G}_{\phi})\simto \pi_{0}(Z(C_{\hat\mathbf{G}}(\phi))^{W_{F}})^{*}$, with $W_{F}$
acting through the canonical action of point iii) in \ref{extensions}. Now the
desired map is induced by the inclusion 
$Z(\hat\mathbf{G})^{W_{F}}\subset Z(C_{\hat\mathbf{G}}(\phi))^{W_{F}}$.

The role of this map is the following. If the group $\mathbf{G}$ is quasi-split, 
the factor category $\Rep_{\phi}(G)$ is expected to be equivalent to the direct product of the
unipotent factors of each $\mathbf{G}_{\phi}^{\tau}$ for $\tau\in \ker(h_{\phi})$. More
generally, for a pure inner form $\mathbf{G}^{\alpha}$ of a quasi-split $\mathbf{G}$
associated to some $\alpha\in H^{1}(F,\mathbf{G})$,
the factor category $\Rep_{\phi}(G^{\alpha})$ is expected to be equivalent to the direct product of the
unipotent factors of each $\mathbf{G}_{\phi}^{\tau}$ for $h_{\phi}(\tau)=\alpha$.

All these connected reductive  $F$-groups $\mathbf{G}_{\phi}^{\tau}$ share the same
$L$-group, which we denote by  $^{L}\mathbf{G}_{\phi}$. As usual, it is defined ``up to
inner automorphism''. To fix ideas, let us choose
an \'epinglage   $\varepsilon$ of $C_{\hat\mathbf{G}}(\phi)$.
Then, as a model for the $L$-group we can take
\ini\begin{equation}
^{L}\mathbf{G}_{\phi}= C_{\hat\mathbf{G}}(\phi) \rtimes_{\alpha_{\phi}^{\varepsilon}} W_{F}\label{Lgpcon}
\end{equation}
 where
the action $\alpha_{\phi}^{\varepsilon}$ is obtained from the canonical outer action via the
splitting ${\rm Out}(C_{\hat\mathbf{G}}(\phi))\injo \Aut(C_{\hat\mathbf{G}}(\phi))$
associated to $\varepsilon$. The $L$-group $^{L}\mathbf{G}_{\phi}$ is an extension of $W_{F}$ by $C_{\hat\mathbf{G}}(\phi)$
but it is not a priori clear whether it is isomorphic to the extension $\wt
C_{\hat\mathbf{G}}(\phi)$. More precisely, let $\wt\NC(\varepsilon)$ be the stabilizer
of $\varepsilon$ in $\wt C_{\hat\mathbf{G}}(\phi)$ acting by conjugacy on
$C_{\hat\mathbf{G}}(\phi)$. Then $\wt\NC(\varepsilon)$ is an extension of $W_{F}$ by the
center $Z(C_{\hat\mathbf{G}}(\phi))$, and we see that
\begin{lem}
  The extensions $\wt C_{\hat\mathbf{G}}(\phi)$ and $^{L}\mathbf{G}_{\phi}$ of $W_{F}$ by
  $C_{\hat\mathbf{G}}(\phi)$ are isomorphic if and only if the class
  $[\wt\NC(\varepsilon)]$ in $H^{2}(W_{F}/K_{F},Z(C_{\hat\mathbf{G}}(\phi)))$ vanishes.
\end{lem}

Note that the class always vanishes if $K_{F}=I_{F}$ since $W_{F}/I_{F}=\ZM$.
We give more details in section \ref{sec:unip-fact}.

\alin{More $F$-groups in the general case} \label{Gphinonconnected}
We now propose a construction without assuming
that  $C_{\hat\mathbf{G}}(\phi)$ is connected.
Let us take up the
split exact sequence
$$ C_{\hat\mathbf{G}}(\phi)\injo \wt
C_{\hat\mathbf{G}}(\phi)/\phi(K_{F})\twoheadrightarrow W_{F}/K_{F}$$
of paragraph \ref{extensions}.
Put $\pi_{0}(\phi):=\pi_{0}(C_{\hat\mathbf{G}}(\phi))$ and $\wt\pi_{0}(\phi):=\wt C_{\hat\mathbf{G}}(\phi)/C_{\hat\mathbf{G}}(\phi)^{\circ}\phi(K_{F})$. Hence we have a split exact
sequence
$$ \pi_{0}(\phi)\injo \wt\pi_{0}(\phi) \twoheadrightarrow W_{F}/K_{F}$$ and 
a possibly non split exact sequence
$$ C_{\hat\mathbf{G}}(\phi)^{\circ}\injo \wt C_{\hat\mathbf{G}}(\phi)/\phi(K_{F}) \twoheadrightarrow \wt\pi_{0}(\phi).$$
Conjugation by any set-theoretic section of the above sequence gives a well-defined
``outer action''
$$\wt\pi_{0}(\phi)\To{}{\rm Out}(C_{\hat\mathbf{G}}(\phi)^{\circ})$$
that factors over a finite quotient of $\wt\pi_{0}(\phi)$. Let
$\mathbf{G}_{\phi}^{\rm split,\circ}$ denote a split group over $F$ which is dual to
$C_{\hat\mathbf{G}}(\phi)^{\circ}$, and fix a section 
$${\rm Out}(C_{\hat\mathbf{G}}(\phi)^{\circ})={\rm Out}(\mathbf{G}_{\phi}^{\rm split,\circ})\To{}
{\rm Aut}(\mathbf{G}_{\phi}^{\rm split,\circ})$$
(i.e. fix an \'epinglage of $\mathbf{G}_{\phi}^{\rm split,\circ}$).
Then we can form the non-connected reductive $F$-group 
$$ \mathbf{G}_{\phi}^{\rm split}:= \mathbf{G}_{\phi}^{\rm split,\circ}\rtimes
\pi_{0}(\phi)$$
which has an action by algebraic $F$-group automorphisms 
$$ \wt\pi_{0}(\phi) \To{\theta} {\rm Aut}(\mathbf{G}_{\phi}^{\rm split})$$
(here $\wt\pi_{0}(\phi)$ acts on $\pi_{0}(\phi)$ by conjugation).
Then,  any continuous section 
$$\sigma:\, W_{F}/K_{F}\To{}\wt\pi_{0}(\phi)$$ (for the topology
induced from $\Gamma_{F}$ on $W_{F}$ and the discrete topology on $\pi_{0}(\phi)$) provides an
$F$-form $ \mathbf{G}_{\phi}^{\sigma}$ 
of $\mathbf{G}_{\phi}^{\rm split}$ such that the action of $W_{F}$ on 
$\mathbf{G}_{\phi}^{\sigma}(\o F)= \mathbf{G}_{\phi}^{\rm split}(\o F)$ is the natural
action twisted by $\theta\circ\sigma$.
 The unit component $\mathbf{G}_{\phi}^{\sigma,\circ}$ is 
quasi-split over $F$  and we have
 $$\mathbf{G}_{\phi}^{\sigma}(F)=
\mathbf{G}_{\phi}^{\sigma,\circ}(F)\rtimes \pi_{0}(\phi)^{\sigma(W_{F})}.$$
Moreover, if $c\in\pi_{0}(\phi)$, conjugation by $(1,c)$ in $\mathbf{G}^{\rm
  split}_{\phi}(\o F)$ induces an $F$-isomorphism $\mathbf{G}_{\phi}^{\sigma}\simto
\mathbf{G}_{\phi}^{\sigma^{c}}$, so that the isomorphism class of $\mathbf{G}_{\phi}^{\sigma}$ over $F$ 
only depends on the image of $\sigma$ in $\o\Sigma(W_{F}/K_{F},\tilde\pi_{0}(\phi))$.

More generally, for any continuous section 
$$\tau:W_{F}/K_{F}\To{}  \mathbf{G}_{\phi}^{\rm split,\circ}(\o F^{K_{F}})\rtimes
\wt\pi_{0}(\phi)$$
(again, here we use the topology of $W_{F}$
induced from that of $\Gamma_{F}$)
 we get an $F$-form $\mathbf{G}_{\phi}^{\tau}$ of $\mathbf{G}_{\phi}^{\rm split}$, which depends only on the
image of $\tau$ in 
$\o\Sigma(W_{F}/K_{F},\mathbf{G}_{\phi}^{\circ,\rm split}(\o F^{K_{F}})\rtimes
\wt\pi_{0}(\phi))$. 
In order to better organize this collection of $F$-groups, we
use the projection $\tau\mapsto \sigma$ :
$$ \o\Sigma\left(W_{F}/K_{F},\mathbf{G}_{\phi}^{\rm split,\circ}(\o F^{K_{F}})\rtimes
  \wt\pi_{0}(\phi)\right)
\To{}  \o\Sigma\left(W_{F}/K_{F},\wt\pi_{0}(\phi)\right)$$
to get a partition
\begin{eqnarray*}
\o\Sigma(\phi):= \o\Sigma\left(W_{F}/K_{F},\mathbf{G}_{\phi}^{\rm split,\circ}(\o F^{K_{F}})\rtimes \wt\pi_{0}(\phi)\right)
&=&\bigsqcup_{\sigma\in \o\Sigma(W_{F},\wt\pi_{0}(\phi))}
H^{1}\left(W_{F}/K_{F},\mathbf{G}_{\phi}^{\sigma,\circ}(\o F^{K_{F}})\right)\\
&=&\bigsqcup_{\sigma\in \o\Sigma(W_{F},\wt\pi_{0}(\phi))}
H^{1}\left(F,\mathbf{G}_{\phi}^{\sigma,\circ}\right)
\end{eqnarray*}
For the second line, we use
that $H^{1}\left(W_{F}/K_{F},\mathbf{G}_{\phi}^{\sigma,\circ}(\o F^{K_{F}})\right)
=H^{1}\left(W_{F},\mathbf{G}_{\phi}^{\sigma,\circ}(\o F)\right)$
due to  the fact that $H^{1}(K_{F},\mathbf{G}^{\sigma,\circ}(\o F))=\{1\}$,
and we use that
$H^{1}\left(W_{F},\mathbf{G}_{\phi}^{\sigma,\circ}(\o
  F)\right)=H^{1}\left(\Gamma_{F},\mathbf{G}_{\phi}^{\sigma,\circ}(\o F)\right)$ due to
our non-standard choice of topology on $W_{F}$.

Hence, if $\tau$ is mapped to $\sigma$, the group $\mathbf{G}_{\phi}^{\tau}$ is an inner form
of $\mathbf{G}_{\phi}^{\sigma}$ ``coming from the unit component''. However, the
collection of all $\mathbf{G}_{\phi}^{\tau}$'s should be viewed as a single ``pure inner
class'' of (possibly non-connected) reductive groups. As in the case of connected
centralizers, we need a map
$$ \o\Sigma(\phi)\To{} H^{1}(F,\mathbf{G})$$ 
in order to shrink this collection of groups. We define it as the coproduct of the maps 
$H^{1}(F,\mathbf{G}_{\phi}^{\sigma,\circ})\To{} H^{1}(F,\mathbf{G})$ which are dually
induced by inclusions 
$$ Z(\hat\mathbf{G})^{W_{F}}\subset Z(C_{\hat\mathbf{G}}(\phi)^{\circ})^{\wt\pi_{0}(\phi)}
\subset Z(C_{\hat\mathbf{G}}(\phi)^{\circ})^{\sigma(W_{F})} = Z(\wh{\mathbf{G}_{\phi}^{\sigma,\circ}})^{W_{F}}.$$
  
Finally we define a common $L$-group for all these non-connected groups. Namely we put 
\ini\begin{equation}
^{L}\mathbf{G}_{\phi}:=(C_{\hat\mathbf{G}}(\phi)^{\circ}
\rtimes_{\alpha_{\phi}^{\varepsilon}} \wt\pi_{0}(\phi))\times_{W_{F/K_{F}}}W_{F}\label{Lgpnoncon}
\end{equation}
where $\varepsilon$ is an \'epinglage, and $\alpha_{\phi}^{\varepsilon}$ is the action
obtained from the canonical outer action thanks to this \'epinglage.
The main inconvenience of this $L$-group is that it might not be always isomorphic to the
extension $\wt C_{\hat\mathbf{G}}(\phi)$. 
More precisely,
let $\wt\NC(\varepsilon)$ be the stabilizer
of $\varepsilon$ in $\wt C_{\hat\mathbf{G}}(\phi)$ acting by conjugacy on
$C_{\hat\mathbf{G}}(\phi)$. Then $\wt\NC(\varepsilon)$ is an extension
of $\wt\pi_{0}\times_{W_{F}/K_{F}}W_{F}$  by the
center $Z(C_{\hat\mathbf{G}}(\phi)^{\circ})$, and we see that
\begin{lem}
  The extensions $\wt C_{\hat\mathbf{G}}(\phi)$ and
  $^{L}\mathbf{G}_{\phi}$ of $\wt\pi_{0}\times_{W_{F}/K_{F}}W_{F}$ by 
  $C_{\hat\mathbf{G}}(\phi)^{\circ}$ are isomorphic if and only if the class
  $[\wt\NC(\varepsilon)]$ in
  $H^{2}(\wt\pi_{0}(\phi),Z(C_{\hat\mathbf{G}}(\phi))^{\circ})$
  vanishes.
\end{lem}
Therefore, a necessary condition for $^{L}\mathbf{G}_{\phi}\simeq
\wt C_{\hat\mathbf{G}}(\phi)$ is
that $\pi_{0}(\phi)$ has a lifting in $C_{\hat\mathbf{G}}(\phi)$ which fixes
$\varepsilon$. When $K_{F}=I_{F}$, this condition is sufficient since $W_{F}/I_{F}=\ZM$.

\alin{An example} \label{examplesl_2}
Assume $p$ is odd, $K_{F}=I_{F}$ and let $\varepsilon$ denote the unique
non-trivial
quadratic continuous character of $I_{F}$. Then consider the group $\mathbf{G}=\SL_{2}$ and the
parameter $\phi :
I_{F}\To{}\PGL_{2}(\oQl)\times W_{F}$ that takes $i$ to 
$\left(
\left(\begin{array}{cc}
  \varepsilon(i) & 0  \\ 0 & 1
\end{array}\right)
,i\right)$.
Then $C_{\hat\mathbf{G}}(\phi)$ is the normalizer $\hat\mathbf{N}$ of the diagonal torus
$\hat\mathbf{T}$ of $\hat\mathbf{G}$. Denote by $s$ the element 
$\left(\begin{array}{cc}
  0 & 1  \\ 1 & 0
\end{array}\right)$, of order $2$.
We thus have
$\mathbf{G}_{\phi}^{\rm split}=\mathbb{G}_{m}\rtimes \{1,s\}$ with $s$ acting by the
inverse map. Moreover we have $\wt\pi_{0}(\phi)=\{1,s\}\times W_{F}/I_{F}$, so that 
$\o\Sigma(W_{F}/K_{F},\wt\pi_{0}(\phi))$ has $2$ elements $\sigma_{0},\sigma_{1}$, with
$\sigma_{0}$ the trivial morphism. Clearly we have
$\mathbf{G}_{\phi}^{\sigma_{0}}=\mathbf{G}_{\phi}^{\rm split}=\GM_{m}\rtimes\{1,s\}$. On
the other hand we have 
$\mathbf{G}_{\phi}^{\sigma_{1}}=\mathbf{U}(1)\rtimes\{1,s\}$ where $\mathbf{U}(1)$ is the
group of norm $1$ elements in the unramified quadratic extension of $F$ and $s$ again acts
by the inverse map. Further, $\o\Sigma(\phi)= H^{1}(F,\GM_{m})\sqcup
H^{1}(F,\mathbf{U}(1))$ has $3$ elements, $\tau_{0}, \tau_{1,1}$ and $\tau_{1,2}$. We
compute that
$$ \mathbf{G}_{\phi}^{\tau_{0}}=\GM_{m}\rtimes \{1,s\},
\,\,\hbox{ and }\,\,
 \mathbf{G}_{\phi}^{\tau_{1,1}}=
 \mathbf{G}_{\phi}^{\tau_{1,2}}=\mathbf{U}(1)\rtimes \{1,s\}.$$
Observe also that in this case we have  $^{L}\mathbf{G}_{\phi}\simeq \wt C_{\hat\mathbf{G}}(\phi)$.

\subsection{Unipotent factorizations of a $K_{F}$-parameter} \label{sec:unip-fact}

We keep the general setup  of the previous section, consisting of a
connected reductive $F$-group $\mathbf{G}$ and an admissible parameter  $\phi:
K_{F}\To{}{^{L}\mathbf{G}}$.
We have associated an $L$-group $^{L}\mathbf{G}_{\phi}$ to $\phi$, see (\ref{Lgpcon}) and
(\ref{Lgpnoncon}), after fixing an \'epinglage $\varepsilon$ of $C_{\hat\mathbf{G}}(\phi)^{\circ}$.
We will denote by $1 : K_{F}\To{} {^{L}\mathbf{G}}_{\phi}$ the morphism that takes $k\in
  K_{F}$ to $(1,k)$.


\begin{DEf}
 A \emph{strict} unipotent factorization of $\phi$ is a morphism of $L$-groups
$\xi:\,^{L}\mathbf{G}_{\phi}\To{}{^{L}\mathbf{G}}$ that extends the inclusion 
$C_{\hat\mathbf{G}}(\phi)^{\circ}\injo \hat\mathbf{G}$ and satisfies $\xi\circ 1=\phi$. 
 Two such
 factorizations are called equivalent if 
 they are conjugate 
by some element $\hat g\in \hat\mathbf{G}$.
\end{DEf}

\begin{prop} Suppose that $C_{\hat\mathbf{G}}(\phi)$ is \emph{connected}. In the following
  statements, we use the canonical action of $W_{F}/K_{F}$ on 
$Z(C_{\hat\mathbf{G}}(\phi))$.
\begin{enumerate}
\item The map $\xi\mapsto \varphi:=\xi_{|W}$ sets up a bijection between
$$\{\hbox{strict unipotent factorizations of } \phi\} \,\hbox{ and }$$
$$\{\hbox{parameters }\varphi:\,W_F\To{}{^{L}\mathbf{G}} \hbox{ that
  extend } \phi \hbox{ and  such that $\alpha_{\varphi}$ preserves }
\varepsilon.\}$$
\item Multiplication of cocycles 
turns the second set of
  point i) into a torsor over $ Z^{1}(W_{F}/K_{F},Z(\hat\mathbf{G}_{\phi}))$.
  \item The set of equivalence
 classes of strict unipotent factorizations of $\phi$ is a torsor 
 over $H^{1}(W_{F}/K_{F},Z(C_{\hat\mathbf{G}}(\phi)))$.
  \item There is an obstruction element $\beta_{\phi}\in H^{2}(W_{F}/K_{F},Z(C_{\hat\mathbf{G}}(\phi)))$
which vanishes if and only if $\phi$ admits a strict unipotent
    factorization.
 \end{enumerate}
\end{prop}

\begin{proof}
i) From the equality $\xi(\hat c,w)=\hat c.\varphi(w)$, we see that
the map is well-defined. To prove it is a bijection it suffices to
check that the inverse map $\varphi\mapsto \xi: (\hat c,w)\mapsto
\hat c.\varphi(w)$ is well-defined too.
But 
 any extension $\varphi$ of $\phi$ to $W_{F}$ leads to a factorization 
\ini \begin{equation}
\xymatrix{
\phi:\,\, K_{F} \ar[rrr]^-{i\mapsto (1,i)} &&&
 C_{\hat\mathbf{G}}(\phi)\rtimes_{\alpha_{\varphi}} W_{F} \ar[rrr]^-{(\hat c,w)\mapsto \hat c.\varphi(w)} &&&
{^{L}\mathbf{G}} 
}.\label{factophi}
\end{equation}

If $\varphi$ preserves the \'epinglage $\varepsilon$, then the semi-direct product in the middle is
${^{L}\mathbf{G}_{\phi}}$  and
the map on the right hand side is therefore a strict unipotent factorization.


ii) Suppose the second set of point i) is not empty, and let $\varphi$
be an element in this set. Then any other $\varphi'$ in this set has
the form $\varphi'(w)=\eta(w)\varphi(w)$ for some $\eta\in
Z^{1}_{\alpha_{\varphi}}(W_{F},C_{\hat\mathbf{G}}(\phi))$. Since
$\varphi_{|K_{F}}=\varphi'_{|K_{F}}$ we have in fact $\eta\in
Z^{1}_{\alpha_{\varphi}}(W_{F}/K_{F},C_{\hat\mathbf{G}}(\phi))$. Moreover, since both
$\alpha_{\varphi}$ and $\alpha_{\varphi'}$ preserve $\epsilon$ and induce the same outer
automorphism, ${\rm  Int}_{\eta(w)}$ has to be trivial, that is, 
$\eta\in Z^{1}_{\alpha_{\varphi}}(W_{F}/K_{F},Z(C_{\hat\mathbf{G}}(\phi)))
= Z^{1}(W_{F}/K_{F},Z(C_{\hat\mathbf{G}}(\phi)))$ (the action on
$Z(C_{\hat\mathbf{G}}(\phi))$ is canonical).
This shows that, if non empty, the second set in point i) is a (obviously principal)
homogeneous set over $Z^{1}(W_{F}/K_{F},Z(C_{\hat\mathbf{G}}(\phi)))$.

iii) Suppose both $\xi$ and $\xi'= {\rm Int}_{\hat g}\circ \xi$, with $\hat g\in
\hat\mathbf{G}$,  are strict unipotent factorizations of $\phi$. Since
$\xi_{|K_{F}}=\xi'_{|K_{F}}=\phi$, we must have 
$\hat g\in C_{\hat\mathbf{G}}(\phi)$. But since also
$\xi_{|C_{\hat\mathbf{G}}(\phi)}=\xi'_{|C_{\hat\mathbf{G}}(\phi)}=\phi$, we have 
$\hat g\in Z(C_{\hat\mathbf{G}}(\phi))$. 
Then we see that $\xi'_{|W} = \eta. \xi_{|W}$, where $\eta$ is the boundary cocycle
$w\mapsto (\hat g .{^{w}(\hat g)^{-1}})$. Hence point iii) follows from ii) and i).

iv) This is Lemma \ref{Gphiconnected}. Here is a more detailed argument.
Start with an arbitrary extension $\varphi$ of $\phi$ to $W_{F}$. We need to investigate
the existence of a cocycle
$\eta\in Z^{1}_{\alpha_{\varphi}}(W_{F}/K_{F},C_{\hat\mathbf{G}}(\phi))$ such that
$\eta\varphi$ fixes the \'epinglage $\varepsilon$.
This is equivalent to asking that 
$\alpha_{\phi}^{\varepsilon}(w) = {\rm  Int}_{\eta(w)} \circ \alpha_{\varphi}(w)$ for all
$w\in W_{F}$. 
But for $w\in W_{F}/K_{F}$, there is a unique element $\beta_{\varphi}^{\varepsilon}(w)\in
C_{\hat\mathbf{G}}(\phi)_{\rm ad}$ such that 
$\alpha_{\phi}^{\varepsilon}(w) = {\rm  Ad}_{\beta_{\varphi}^{\varepsilon}(w)} \circ
\alpha_{\varphi}(w)$. 
Unicity insures that the map 
$w\mapsto \beta_{\varphi}^{\varepsilon}(w)$ lies in
$Z^{1}_{\alpha_{\varphi}}(W_{F}/K_{F},C_{\hat\mathbf{G}}(\phi)_{\rm ad})$,
and the existence of $\eta$ as above is equivalent to the vanishing of
the image $\beta_{\phi}$ of $\beta_{\varphi}^{\varepsilon}$ by the boundary map
$$ H^{1}_{\alpha_{\varphi}}(W_{F}/K_{F},C_{\hat\mathbf{G}}(\phi)_{\rm ad})
\To{} H^{2}(W_{F}/K_{F},Z(C_{\hat\mathbf{G}}(\phi))).$$
\end{proof}

\begin{rema}\label{Rkstrictunip}
  Here is the significance of point iii) in terms of transfer of representations. Assume that a
  strict unipotent factorization $\xi$ of $\phi$ exists. Then the transfer map dual to $\xi$, from the set of
  $L$-packets of $G_{\phi}$ to that of $G$, only depends on the $Z(\hat\mathbf{G}_{\phi})$-conjugacy
  class of $\xi$. If we change $\xi$ to $\eta.\xi$ for $\eta \in
  H^{1}(W_{F}/K_{F},Z(\hat\mathbf{G}_{\phi}))$ then the transfer map is twisted by the character of
  $G_{\phi}$ associated to $\eta$ in \cite[10.2]{BorelCorvallis} (this character is
  unramified if $K_{F}=I_{F}$ or has level $0$ for $K_{F}=P_{F}$).
\end{rema}

\begin{coro}
  When $K_{F}=I_{F}$, any parameter $\phi$ admits a strict unipotent factorization (provided
  $C_{\hat\mathbf{G}}(\phi)$ is connected).
\end{coro}
\begin{proof}
  In this case $W_{F}/K_{F}=\ZM$ , so
$H^{2}(W_{F}/K_{F},A)=0$ for any $\ZM[W_{F}/K_{F}]$-module $A$.
\end{proof}

\begin{rema}
  In the non-connected case, points iii) and iv) remain true with the pair
  $(W_{F}/K_{F},Z(C_{\hat\mathbf{G}}(\phi)))$ replaced by $(\wt\pi_{0}(\phi),
  Z(C_{\hat\mathbf{G}}(\phi)^{\circ}))$. However, we will not use it in this paper.
\end{rema}

In the sequel, we will also encounter ``non-strict'' unipotent factorizations.

\begin{DEf}
A  \emph{unipotent factorization} is a pair $(\mathbf{H},\xi)$ consisting of a \emph{connected}
  reductive $F$-group and a morphism of $L$-groups
  $^{L}\mathbf{H}\To{\xi}{^{L}\mathbf{G}}$ such that
  \begin{itemize}
  \item $\phi$ is $\hat\mathbf{G}$-conjugate to $\xi\circ 1$ with $1$ the trivial
    parameter $k\in K_{F}\mapsto (1,k)\in {^{L}\mathbf{H}}$.
  \item $\xi$ induces an isomorphism $\hat\mathbf{H}\simto C_{\hat\mathbf{G}}(\phi)$.
  \end{itemize}
\end{DEf}


Let us make explicit the relation between unipotent factorizations as above 
and strict ones.
If $\iota$ is an \'epinglage-preserving $W_{F}$-equivariant isomorphism
$\hat\mathbf{H}\simto \hat\mathbf{G}_{\phi}$, we denote by 
$^{L}\iota:= \iota\times \id_{W_{F}}:\,{^{L}\mathbf{H}}\simto {^{L}\mathbf{G}_{\phi}}$ the
associated isomorphism of $L$-groups.

\begin{prop} \label{propunipfacto}
 If $(\mathbf{H},\xi)$ is a unipotent factorization of $\phi$, there are an
    \'epinglage-preserving $W_{F}$-equivariant isomorphism 
$\iota:\, \hat\mathbf{H}\simto {\hat\mathbf{G}_{\phi}}$ and a strict
    unipotent factorization $\xi'$ such that $\xi$ is $\hat\mathbf{G}$-conjugate to $\xi'\circ
    {^{L}\iota}$. Moreover $\iota$ is unique and
 $\xi'$ is unique up to equivalence.
\end{prop}
\begin{proof}
 Conjugating $\xi$ under $\hat\mathbf{G}$ we may assume that $\xi\circ 1=\phi$. Conjugating
further $\xi$ under $C_{\hat\mathbf{G}}(\phi)$, we may assume also that $\xi$ takes the given \'epinglage on
$\hat\mathbf{H}$ to $\varepsilon$. Now let $\varphi'$ be the trivial parameter
$W_{F}\To{}{^{L}\mathbf{H}}$ and put $\varphi:=\xi\circ\varphi'$. Then $\xi_{|\hat\mathbf{H}}$ is
$W_{F}$-equivariant for the actions $\alpha_{\varphi'}$ on $\hat\mathbf{H}$ and $\alpha_{\varphi}$ on
$C_{\hat\mathbf{G}}(\phi)$. But $\alpha_{\varphi'}$ is the natural action on $\hat\mathbf{H}$ and
preserves the given \'epinglage, hence $\alpha_{\varphi}$ preserves $\varepsilon$. Therefore
$\iota:=\xi_{|\hat\mathbf{H}}$ is an \'epinglage-preserving $W_{F}$-equivariant isomorphism 
$\hat\mathbf{H}\simto {\hat\mathbf{G}_{\phi}}$ and $\xi':=\xi\circ{^{L}\iota}^{-1}$ is a strict
unipotent factorization, whence the existence statement. Unicity of $\iota$ is clear, and any
other $\xi''$ has to be both $\hat\mathbf{G}_{\phi}$-conjugate to $\xi'$ and strict, hence is
$Z(\hat\mathbf{G}_{\phi})$-conjugate to $\xi'$.
\end{proof}

\begin{rak}
We see in particular that a general
  unipotent factorization $(\mathbf{G}_{\phi},\xi')$ is equivalent to 
the composition of a strict one
  with an ``outer'' $W_{F}$-equivariant automorphism of $\hat\mathbf{G}_{\phi}$. 
\end{rak}

\subsection{Restriction of scalars}
We consider here a reductive group $\mathbf{G}$ over $F$ of the form
$\mathbf{G}={\rm Res}_{F'|F}\mathbf{G'}$ for some reductive group
$\mathbf{G'}$ over some extension field $F'$ of $F$. We then have the
following relationship between their dual groups equipped with Weil
group actions :
\def\hg{{\hat g}}
\def\hh{{\hat h}}
$$ \hat\mathbf{G}={\rm Ind}_{W_{F'}}^{W_{F}}\hat\mathbf{G}'
=\left\{\hg:\, W_{F}\To{}\hat\mathbf{G}',\, \forall (w',w)\in W_{F'}\times
W_{F}, \hg(w'w)={^{w'}(\hg(w))}\right\}$$
where we have denoted with an exponent $^{w'}$ the action of $W_{F'}$
on $\hat\mathbf{G}'$ and we let $v\in W_{F}$ act on $\mathbf{G}$ by 
$(v\hg)(w):=\hg(wv)$.

\ali  We still denote by  $K_{F}$  a closed normal subgroup of $W_{F}$,
and we put $K_{F'}:=W_{F'}\cap K_{F}$.
Note that if $K_{F}$ is one of the groups $I_{F}$, $I_{F}^{(\ell)}$ or $P_{F}$,
we have respectively $K_{F'}=I_{F'}$, $I_{F'}^{(\ell)}$ or $P_{F'}$.
There is a natural map
on continuous cocycles
$$ Z^{1}(K_{F},\hat\mathbf{G}) \To{} Z^{1}(K_{F'},\hat\mathbf{G}')$$ that takes the cocycle
$(\hg_{\gamma})_{\gamma\in K_{F}}$ to the cocycle $(\hg_{\gamma'}(1))_{\gamma'\in K_{F'}}$. We will
call it the ``Shapiro map''. It is compatible with coboundary relation thus induces a map
$$ H^{1}(K_{F},\hat\mathbf{G}) \To{} H^{1}(K_{F'},\hat\mathbf{G}').$$
Shapiro's lemma asserts that when $K_{F}=W_{F}$ the map on $Z^{1}$ is onto while the map on $H^{1}$
is a bijection. 
From the definition of Shapiro's map we have two commutative diagrams
$$
\xymatrix{ 
 Z^{1}(W_{F},\hat\mathbf{G}) \ar@{->>}[r]\ar[d]_{{\rm res}} &
 Z^{1}(W_{F'},\hat\mathbf{G}') \ar[d]_{{\rm res}} \\
 Z^{1}(K_{F},\hat\mathbf{G}) \ar[r] &
 Z^{1}(K_{F'},\hat\mathbf{G}') 
}
\hbox{ and }
\xymatrix{ 
 H^{1}(W_{F},\hat\mathbf{G}) \ar[r]^{\sim}\ar[d]_{{\rm res}} &
 H^{1}(W_{F'},\hat\mathbf{G}') \ar[d]_{{\rm res}} \\
 H^{1}(K_{F},\hat\mathbf{G}) \ar[r] &
 H^{1}(K_{F'},\hat\mathbf{G}') 
}
$$

\begin{lemme}\label{lemmaShapiro}
  The right hand square is cartesian.
\end{lemme}
\begin{proof}
  The two diagrams are transitive with respect to intermediate field extensions. Applying this to
  the extension $F''$ defined by $W_{F''}=W_{F'}K_{F}$, we see it
  is enough to prove the claim in the following 2 cases : a) $K_{F}=K_{F'}$
 or b) $W_{F}=K_{F}W_{F'}$.

In case b), we have $\hat\mathbf{G}={\rm Ind}_{K_{F'}}^{K_{F}}\hat\mathbf{G}'$ as groups with
$K_{F}$-action, so that the bottom map of our diagram is also an isomorphism and the claim is clear.

In case a) we have to prove that for any two cocycles 
$(\hg_{w})_{w\in W_{F}}$, and $(\hh_{w})_{w\in  W_{F}}$ in $Z^{1}(W_{F},\hat\mathbf{G})$, we have
$$ \left(\forall \gamma\in K_{F}=K_{F'},\, \hg_{\gamma}(1)=\hh_{\gamma}(1) \right)
 \Rightarrow
\left((\hg_{\gamma})_{\gamma\in K_{F}} = (\hh_{\gamma})_{\gamma\in  K_{F}}
\hbox{ in } H^{1}(K_{F},\hat\mathbf{G})\right).
$$
Let us fix a set $\{v_{1}=1,\ldots, v_{r}\}$ of representatives of left $W_{F'}$-cosets in $W_{F}$ and
let us denote by $\hat k$ the unique element of $\hat\mathbf{G}$ such that $\hat k(v_{i})=
\hh_{v_{i}}(1)^{-1}\hg_{v_{i}}(1)$ for all $i$.

The cocycle property tells us that $\hg_{\gamma}(v)=\hg_{v}(1)^{-1}\hg_{v\gamma}(1)$ for all
$v\in W_{F}$, hence also $\hg_{\gamma}(v)=\hg_{v}(1)^{-1}\cdot\hg_{v\gamma v^{-1}}(1)\cdot
{^{v\gamma v^{-1}}(\hg_{v}(1))}$ and the same for $\hh$.
 Since by hypothesis we have $\hg_{v\gamma v^{-1}}(1)=\hh_{v\gamma v^{-1}}(1)$,
this implies that for each $i$ we have

$$\hh_{\gamma}(v_{i})
=\hat k(v_{i}) \cdot\hg_{\gamma}(v_{i}) \cdot{^{v_{i}\gamma v_{i}^{-1}}\hat k}(v_{i})^{-1}
=\hat k(v_{i}) \cdot\hg_{\gamma}(v_{i}) \cdot\hat k(v_{i}\gamma)^{-1}
$$ 
Then, for any $v\in W_{F}$, writing uniquely $v=v'v_{i}$ with $v'\in W_{F'}$ we get
$$\hh_{\gamma}(v)= {^{v'}\hh_{\gamma}(v_{i})} = 
{^{v'}\hat k(v_{i})} \cdot {^{v'}\hg_{\gamma}(v_{i})} \cdot {^{v'}\hat k}(v_{i}\gamma)^{-1}
=\hat k(v)  \cdot\hg_{\gamma}(v) \cdot \hat k(v\gamma)^{-1}.$$
Since $\hat k (v\gamma) = (\gamma\hat k)(v)$, this shows that $(\hh_{\gamma})_{\gamma\in K_{F}}$ is
cohomologous to $(\hg_{\gamma})_{\gamma\in K_{F}}$.
\end{proof}

Recall that, by definition, the set of admissible $K_{F}$-parameters for $\mathbf{G}$ is the set of
continous sections
$\phi: K_{F}\To{}{^{L}\mathbf{G}}$ such that, writing
$\phi(\gamma)=(\hat\phi_{\gamma},\gamma)$, we have
$$(\hat\phi_{\gamma})_{\gamma\in K_{F}}\in {\rm  Image}\left(\Phi(\mathbf{G},\oQl)\To{\rm res}
H^{1}(W_{F},\hat\mathbf{G}) \To{\rm res}
H^{1}(K_{F},\hat\mathbf{G})\right).$$
Recall also that we have denoted this set by  $\Phi_{\rm inert}(\mathbf{G},\oQl)$,
 $\Phi_{\ell'-\rm inert}(\mathbf{G},\oQl)$ and 
$\Phi_{\rm wild}(\mathbf{G},\oQl)$ according to $K_{F}$ being $I_{F}$, $I_{F}^{(\ell)}$
and $P_{F}$. 

\begin{coro}\label{coroShapiro}
  The Shapiro map induces bijections $\Psi(\mathbf{G},\oQl)\simto\Psi(\mathbf{G'},\oQl)$
  where $\Psi$ denotes either $\Phi_{\rm inert}$ or $\Phi_{\ell'-\rm inert}$ or
  $\Phi_{\rm wild}$.
\end{coro}
\begin{proof}
  This is because the Shapiro bijection $H^{1}(W_{F},\hat\mathbf{G}) \To{}
  H^{1}(W_{F'},\hat\mathbf{G}')$ preserves the admissibility conditions on both sides \cite[8.4]{BorelCorvallis}.
\end{proof}

  Let $\phi: K_{F}\To{} {^{L}\mathbf{G}}$ be an admissible $K_{F}$-parameter 
and let $\phi'$ be its Shapiro mate. Pick a parameter $\varphi:\,W_{F}\To{}
  {^{L}\mathbf{G}}$ that extends $\phi$ and let $\varphi'$ be its
  Shapiro mate. We get an action ${\rm Int}_{\varphi}:\,w\mapsto
{\rm Int}_{\varphi(w)}$  of $W_{F}$ on $\hat\mathbf{G}$, where ${\rm Int}_{\varphi(w)}$ means
conjugation by $\varphi(w)$ inside $^{L}\mathbf{G}$. Similarly,
we have an action ${\rm Int}_{\varphi'}$ of $W_{F'}$ on $\hat\mathbf{G}'$. Now consider the
map
$$\application{}{(\hat\mathbf{G},{\rm Int}_{\varphi})}
{{\rm Ind}_{W_{F'}}^{W_{F}}(\hat\mathbf{G}',{\rm Int}_{\varphi'})
}
{\hg}{\tilde g :\, w\mapsto  [{\rm Int}_{\varphi(w)}(\hg)](1)}$$
It is easily checked that this map is well-defined and is a
$W_{F}$-equivariant isomorphism of
groups. In fact, writing $\varphi(w)= (\hat\varphi_{w},w)$, we have
$\tilde g(w)=\hat\varphi_{w}(1) \hg(w)\hat\varphi_{w}(1)^{-1}$, so
that the inverse isomorphism is given by 
$\hg(w)=\hat\varphi_{w}(1)^{-1} \tilde g(w)\hat\varphi_{w}(1)$.

The centralizer
$C_{\hat\mathbf{G}}(\phi)$ is stable under the action ${\rm Int}_{\varphi}$, 
and in the last section we had denoted the resulting action by $\alpha_{\varphi}$. 
Similarly, 
$C_{\hat\mathbf{G}}(\phi')$ is stable under the action ${\rm Int}_{\varphi'}$.

\begin{lemme}
  The above isomorphism takes $C_{\hat\mathbf{G}}(\phi)$ into 
${\rm Ind}_{W_{F'}}^{W_{F}}(C_{\hat\mathbf{G}'}(\phi'))$ and induces an isomorphism
$$(C_{\hat\mathbf{G}}(\phi),\alpha_{\varphi})
\simto {\rm Ind}_{W_{F'}/K_{F'}}^{W_{F}/K_{F}}(C_{\hat\mathbf{G}'}(\phi'),\alpha_{\varphi'})$$
where we identify  the RHS with the $K_{F}$-invariant functions in 
${\rm Ind}_{W_{F'}}^{W_{F}}(C_{\hat\mathbf{G}'}(\phi'))$. 
Moreover, $\alpha_{\varphi}$ preserves an \'epinglage of $C_{\hat\mathbf{G}}(\phi)^{\circ}$ if and only if
$\alpha_{\varphi'}$  preserves an \'epinglage of $C_{\hat\mathbf{G}'}(\phi')^{\circ}$.
\end{lemme}

\begin{proof}
By definition $C_{\hat\mathbf{G}}(\phi)$ is the subgroup of fixed
points under $K_{F}$ acting on $(\hat\mathbf{G},{\rm
  Int}_{\varphi})$. Hence the above isomorphism carries it to the
subgroup
${\rm  Ind}_{W_{F'}}^{W_{F}}(\hat\mathbf{G}',\varphi')^{K_{F}}$.
However for a function $\tilde g$, being $K_{F}$-invariant means
$\tilde g(w\gamma)=\tilde g(w)$ for all $w\in W_{F}$ and $\gamma\in
K_{F}$. Since $K_{F}$ is
  normal in $W_{F}$ this is equivalent to $\tilde g(\gamma w)=\tilde g(w)$ for
   all $w,\gamma$. Applying this to $\gamma\in K_{F'}$ we get
   that $\tilde g(w)\in C_{\hat\mathbf{G}'}(\phi')$ for all $w$, as claimed.
\def\Ind{{\rm Ind}}

Now if $\varepsilon=(B,T,\{x_{\alpha}\})$ is an \'epinglage of 
$C_{\hat\mathbf{G}}(\phi)$ fixed by $\alpha_{\varphi}$,  evaluation at $1$ in the 
isomorphism of the lemma provides an  \'epinglage of 
$C_{\hat\mathbf{G}'}(\phi')$ fixed by $\alpha_{\varphi'}$.
Conversely,
 let $\varepsilon'=(B',T',(x')_{\alpha'\in \Delta'})$  be an \'epinglage of 
$C_{\hat\mathbf{G}'}(\phi')$ stable by $\alpha_{\varphi'}$. 
Put $B={\rm Ind}_{W_{F'}/K_{F'}}^{W_{F}/K_{F}}(B')$ and 
$T={\rm  Ind}_{W_{F'}/K_{F'}}^{W_{F}/K_{F}}(T')$. This is a Borel pair in the
group ${\rm  Ind}_{W_{F'}/K_{F'}}^{W_{F}/K_{F}}(C_{\hat\mathbf{G}'}(\phi'))$,
with set of simple roots
$\Delta= {\rm Ind}(\Delta')=(W_{F'}/K_{F'})\ba [(W_{F}/K_{F})\times \Delta']. $
For a simple root $\alpha=\overline{(v,\alpha')}$, let $x_{\alpha} : W_{F}/K_{F}\To{}
\Hom(\GM_{a},C_{\hat\mathbf{G}'}(\phi'))$ be the function supported on $W_{F'}v$ given by
$x_{\alpha}(w'v)={^{w'}x_{\alpha'}}=x_{w'\alpha'}$.
The triple $(B,T,(x_{\alpha})_{\alpha\in\Delta})$ is then a $W_{F}$-stable \'epinglage
of $\Ind_{W_{F'}/K_{F'}}^{W_{F}/K_{F}}(C_{\hat\mathbf{G}'}(\phi'))$ which, through the
isomorphism of the lemma, provides an \'epinglage $\varepsilon$ 
of $C_{\hat\mathbf{G}}(\phi)$   fixed by $\alpha_{\varphi}$.
\end{proof}

Suppose now that $C_{\hat\mathbf{G}'}(\phi')$ is connected, or equivalently, that
$C_{\hat\mathbf{G}}(\phi)$ is connected, and let us 
fix \'epinglages $\varepsilon$ and $\varepsilon'$ to build the $L$-groups
$^{L}\mathbf{G}_{\phi}$ and $^{L}\mathbf{G}'_{\phi'}$. 
The following is a translation of the last lemma in the language of the previous section. 

\begin{coro} \label{factoShapiro}
  \begin{enumerate}
  \item We have the following relation between $\mathbf{G}_{\phi}$ and $\mathbf{G}'_{\phi'}$. Denote
    by $F''$ the intermediate extension such that $W_{F''}=W_{F'}K_{F}$ and let $\phi'':K_{F''}\To{}
   {^{L}\mathbf{G}''}$ with $\mathbf{G}''={\rm Res}_{F'|F''}\mathbf{G'}$ be the Shapiro mate of $\phi'$. Then we have
$$ \mathbf{G}_{\phi} \simeq {\rm Res}_{F''|F} \mathbf{G}_{\phi''}''
\,\,\hbox{ and }\,\,
 \mathbf{G}'_{\phi'} \simeq \mathbf{G}''_{\phi''}\times_{F''} F',$$
whence in particular an $L$-homomorphism (unique up to conjugacy)
$$ {^{L}\mathbf{G}_{\phi}} 
\To{\xi_{u}} {^{L}{\rm Res}_{F'|F}(\mathbf{G}'_{\phi'})}$$
which is an isomorphism if $F''=F'$ (\emph{i.e.} $K_{F'}=K_{F}$), while its adjoint
  $^{L}(\mathbf{G}_{\phi}\times_{F}F')\To{}{^{L}\mathbf{G'}_{\phi'}}$ is an isomorphism if $F''=F$
  (\emph{i.e} $W_F=W_{F'}K_{F}$). 
  \item 
$\phi$ admits a strict unipotent factorization if and only if $\phi'$ does. Moreover, if
  $\xi:{^{L}\mathbf{G}_{\phi}}\To{}{^{L}\mathbf{G}}$ is a strict unipotent factorization of
    $\phi$, there are a strict unipotent factorization
    $\xi':{^{L}\mathbf{G}'_{\phi'}}\To{}{^{L}\mathbf{G}'}$ of $\phi'$ and a factorization
$$ \xi:\, {^{L}\mathbf{G}_{\phi}} 
\To{\xi_{u}} {^{L}{\rm Res}_{F'|F}(\mathbf{G}'_{\phi'})} \To{\tilde{\xi'}} {^{L}{\rm
    Res}_{F'|F}(\mathbf{G}')}={^{L}\mathbf{G}}
$$
with $\tilde{\xi'}_{|\widehat{{\rm Res}(\mathbf{G}'_{\phi'})}}= {\rm  Ind}_{W_{F'}}^{W_{F}}(\xi'_{|\hat\mathbf{G}'_{\phi'}})$
and $\tilde{\xi'}_{|W_{F}}$ a Shapiro lift of  $\xi'_{|W_{F'}}$.
\end{enumerate}
\end{coro}

\begin{rema} \label{RkShapiro}
Because of the form taken by $\tilde{\xi'}$,
 the tranfer map 
$$\tilde{\xi'}_{*}:\Phi({\rm Res}_{F'|F}(\mathbf{G}'_{\phi'}),\oQl)\To{}
  \Phi({\rm Res}_{F'|F}(\mathbf{G}'),\oQl)$$
 coincides with the transfer map $\xi'_{*}$ through the
  Shapiro bijections. Moreover, since $\xi_{u}$ induces an isomorphism 
$ \hat\mathbf{G}_{\phi}^{K_{F}}\simto \widehat{{\rm Res}_{F'|F}\mathbf{G}'_{\phi'}}^{K_{F}}$,
the map  $\tilde{\xi'}$ induces an isomorphism 
$C(1)=\widehat{{\rm Res}_{F'|F}\mathbf{G}'_{\phi'}}^{K_{F}}\simto C_{\hat\mathbf{G}}(\phi)$.
\end{rema}

\subsection{Groups of GL-type}\label{sec:groups-gl-type}

Recall that $\mathbf{G}$ is of $\GL$-type if it is isomorphic to a product of groups of the form
${\rm Res}_{F'|F}(\GL_{n})$. For such a group, the local Langlands correspondence for $\GL_{n}$ and
the Shapiro lemma provide a bijection 
$\Irr{\oQl}{\mathbf{G}(F)}\simto \Phi(\mathbf{G},\oQl)$, $\pi\mapsto \varphi_{\pi}$. 
\begin{lemme} \label{lemmeblockQl}
Let $\phi\in \Phi_{\rm inert}(\mathbf{G},\oQl)$. Define 
$\Rep_{\phi}(G)$ as the smallest direct factor of $\Rep_{\oQl}(G)$ which contains all
irreducible $\pi$ such that ${\varphi_{\pi}}_{|I_{F}}\sim\phi$. Then
 $\Rep_{\phi}(G)$ is a Bernstein block of $\Rep_{\oQl}(G)$.
\end{lemme}
\begin{proof} We may assume that $\mathbf{G}={\rm Res}_{F'|F}(\GL_{n})$.
In this case, Lemma \ref{lemmaShapiro} and Corollary
  \ref{coroShapiro} provide us with a Shapiro bijection $\Phi_{\rm
    inert}(\GL_{n},\oQl)\simto\Phi_{\rm inert}(\mathbf{G},\oQl)$, $\phi'\mapsto \phi$,
  such that $\Rep_{\phi'}(\GL_{n}(F'))=\Rep_{\phi}(\mathbf{G}(F))$. So we are reduced
  to the case $\mathbf{G}=\GL_{n}$. In this case, we need to prove that the extensions
  $\varphi$ of $\phi$ fall in a single ``inertial class''. Equivalently, writing
  $\varphi(w)=(\hat\varphi(w),w)$ with $\hat\varphi$ an $n$-dimensional representation, we
  see that we need to prove
  that if two semisimple representations $\hat\varphi,\hat\varphi'$ of $W_{F}$ are isomorphic as
  $I_{F}$-representations, then there are decompositions
  $\hat\varphi=\hat\varphi_{1}\oplus\cdots\oplus\hat\varphi_{r}$ and
  $\hat\varphi'=\hat\varphi'_{1}\oplus\cdots\oplus\hat\varphi'_{r}$
  and unramified characters $\chi_{i}$, $i=1,\ldots, r$, 
  of $W_F$ such that $\hat\varphi'_{i}=\chi_{i}\hat\varphi_{i}$ for all $i=1,\cdots, r$. But
  Clifford theory tells us that the restriction of an irreducible $\hat\varphi$ to $I_{F}$
  has multiplicity one, and that any $W_{F}$-invariant multiplicity one semisimple
  representations of $I_{F}$ extends to a representation of $W_{F}$ which is irreducible
  and unique up to unramified twist. Hence any decomposition 
  $\hat\varphi=\hat\varphi_{1}\oplus\cdots\oplus\hat\varphi_{r}$ into irreducible summands has to be
  matched by a similar decomposition of $\hat\varphi'$ satisfying the desired twisting property.
\end{proof}
 Therefore, we have a parametrization of
 Bernstein blocks of $\Rep_{\oQl}(\mathbf{G}(F))$ by $\Phi_{\rm inert}(\mathbf{G},\oQl)$,
 which moreover is compatible with the Shapiro bijection.
Let us turn to Vign\'eras-Helm blocks. 

\begin{prop}\label{lemmeblockZl}
  Let $\phi\in\Phi_{\ell'-\rm inert}(\mathbf{G},\oQl)$. There is a unique direct factor subcategory
  $\Rep_{\phi}(G)$ of $\Rep_{\oZl}(\phi)$ such that for any $\pi\in\Irr{\oQl}{G}$ we have
  $\pi\in \Rep_{\phi}(G)$ if and only if ${\varphi_{\pi}}_{|I_{F}^{(\ell)}}\sim
  \phi$. Moreover $\Rep_{\phi}(G)$ is a block.
\end{prop}
\begin{proof} Again, we may assume that $\mathbf{G}={\rm Res}_{F'|F}(\GL_{n})$, and using 
 Lemma \ref{lemmaShapiro} and Corollary
  \ref{coroShapiro},
we are reduced
  to the case $\mathbf{G}=\GL_{n}$.
  The Vign\'eras blocks of $\Rep_{\oFl}(\GL_{n}(F))$ are parame\-trized by inertial
  classes of semisimple $\oFl$-representations of $W_{F}$, and the same proof as above
  shows these are in bijection with isomorphism classes of semisimple
  $\oFl$-representations of $I_{F}$ that extend to $W_{F}$. As explained in \ref{Lanparzl},
  the latter are in
  bijection with isomorphism classes  of 
(semisimple) $\oQl$-representations of $I_{F}^{(\ell)}$, via reduction mod
  $\ell$ and restriction. Going through these identifications, considering the
  definition of the ``mod $\ell$ inertial supercuspidal support'' of  a
  $\pi\in\Irr{\oQl}{\GL_{n}(F)}$ in   \cite[Def. 4.10]{HelmBernstein} and applying Theorem
11.8 of \cite{HelmBernstein}, we find that
  $\pi,\pi'\in\Irr{\oQl}{\GL_{n}(F)}$ lie in the same Helm block if and only if
  $\varphi_{\pi}$ and $\varphi_{\pi'}$ have isomorphic restrictions to
  $I_{F}^{(\ell)}$. 
\end{proof}

We thus get a parametrization of Vign\'eras-Helm blocks of
$\Rep_{\oZl}(\mathbf{G}(F))$ by $\Phi_{\ell'-\rm inert}(\mathbf{G},\oQl)$, 
which is compatible with Shapiro bijections.

\medskip

\alin{Assumptions and convention}
In the sequel,  $K_{F}$ will denote one of the subgroups $I_{F}$, $I_{F}^{(\ell)}$ or $P_{F}$ of $W_{F}$. 
 The notation $\Rep_{\phi}(G)$ will denote a block of $\Rep_{\oZl}(G)$ if
$K_{F}=I_{F}^{(\ell)}$, and a block of $\Rep_{\oQl}(G)$ if $K_{F}=I_{F}$. When 
$K_{F}=P_{F}$, it will denote a direct factor of $\Rep_{\oZl}(G)$.
On the other hand, the notation $\Irr{\phi}{G}$ will always denote the set of $\oQl$-irreducible representations in
this block.

\ali
We will denote by $\EC_{F}(\phi',\xi)$ the following statement, that
depends on an admissible $K_{F}$-parameter $\phi' :\, K_{F}\To{}{^{L}\mathbf{G'}}$ and an 
$L$-homomorphism $\xi:\,{^{L}\mathbf{G'}}\To{}{^{L}\mathbf{G}}$ which induces an isomorphism
$C_{\hat\mathbf{G}'}(\phi')\simto C_{\hat\mathbf{G}}(\phi)$,  where $\phi=\xi\circ\phi'$.

\begin{center}
  $\EC_{F}(\phi',\xi)$ : 
  \begin{tabular}{|l}
there is an equivalence of categories $\Rep_{\phi'}(G')\simto\Rep_{\phi}(G)$ \\
  that extends the transfer map $\xi_{*}:\,\Irr{\phi'}{G'}\To{}\Irr{\phi}{G}$.
\end{tabular}
\end{center}

We also denote by $\EC_F(\phi',\xi)^{-}$ the same statement without the condition on the
transfer map.

\begin{exam}\label{exempleequiv}
  Suppose that $\xi$ is a Levi subgroup embedding. Then we can embed $\mathbf{G'}$
  as an $F$-rational Levi subgroup of $\mathbf{G}$ (well-defined up to conjugacy). The assumption
  that $\xi$ induces an isomorphism of centralizers translates into the property that the
 normalizer of the inertial supercuspidal support of any $\pi\in
 \Irr{\phi}{G}$ is contained (up to conjugacy) in 
  $\mathbf{G'}$. In this context, it is known that for any parabolic subgroup $\mathbf{P}$ of
  $\mathbf{G}$ with $\mathbf{G}'$ as a Levi component, the \emph{normalized} parabolic induction
  functor ${\rm Ind}_{P}^{G}$ induces an equivalence of categories $\Rep_{\phi}(G')\simto
  \Rep_{\phi}(G)$  (which is independent of the choice of $P$ up to natural transform).
We refer to \cite[Thm. 12.3]{HelmBernstein} for $\oZl$ coefficients.
We claim that this equivalence is compatible with the transfer $\xi_{*}$. Indeed, using the Shapiro yoga, it is
enough to treat the case of $\mathbf{G}=\GL_{n}$ and $\mathbf{G'}=\prod_{i=1}^{r}\GL_{n_{i}}$, where
$\xi_{*}$ is given in terms of representations by $(\sigma_{1},\cdots,\sigma_{r})\mapsto
\sigma_{1}\oplus\cdots\oplus\sigma_{r}$, which is well known to correspond to normalized parabolic induction in
this context. Therefore $\EC_F(\phi',\xi)$ is satisfied in this setting.
\end{exam}

\alin{Computation of $\mathbf{G}_{\phi}$ when $\mathbf{G}=\GL_{n}$} \label{computGL}
In this case we may write $\phi= \hat\phi\times\id_{W_{F}}$ where $\hat\phi$ is
 an $n$-dimensional semi-simple representation of $K_{F}$ that can be extended to
 $W_{F}$. Our aim is to find a nice extension  $\hat\varphi$ of $\hat\phi$ to $W_{F}$.
 There is a decomposition 
$\hat\phi= \hat\phi_{1}\oplus\cdots\oplus\hat\phi_{r}$, uniquely determined (up to
reordering) by the following properties :
\begin{enumerate}
\item the irreducible constituents of  $\hat\phi_{i}$ form a  $W_{F}$-orbit,
\item $\Hom_{K_{F}}(\hat\phi_{i},\hat\phi_{j})=0$ whenever $i\neq j$.
\end{enumerate}
Since this decomposition is preserved by any extension of $\hat\phi$ to $W_{F}$, 
each $\hat\phi_{i}$ is extendable to $W_{F}$. Putting $n_{i}:=\dim\hat\phi_{i}$, this
means that $\phi$ factors through a Levi subgroup embedding
$\iota:\,(\GL_{n_{1}}\times\cdots\times\GL_{n_{r}})\times W_{F} \injo \GL_{n}\times
W_{F}$. Moreover, by ii) this Levi subgroup contains the centralizer of $\phi$ (in other
words $\iota$ induces an isomorphism of centralizers).

To compute the centralizer, let us write 
$$\hat\phi_{i} = \oQl^{e_{i}}\otimes\left(\bigoplus_{w\in W_{F}/W_{\sigma_{i}}} {^{w}\sigma_{i}}\right)$$
where $\sigma_{i}$ is some irreducible representation of $K_{F}$ and $W_{\sigma_{i}}$ its stabilizer in
$W_{F}$. This decomposition identifies $C_{\hat\mathbf{G}}(\phi)$ with
$\prod_{i=1}^{r}\GL_{e_{i}}(\oQl)^{[W_{F}:W_{\sigma_{i}}]}$. 

\begin{lem}
  Any irreducible representation $\sigma$ of $K_{F}$ can be extended to its normalizer
  $W_{\sigma}$ in $W_{F}$.
\end{lem}
\begin{proof}
The case $K_{F}=I_{F}$ is clear since $W_{F}/I_{F}\simeq\ZM$, so we assume that
$K_{F}\subsetneq I_{F}$.
In this case,
 $W_{\sigma}/K_{F}$ is the semi-direct product of a pro-cyclic group
$(W_{\sigma}\cap I_{F})/K_{F}$ by a copy of $\ZM$ acting by multiplication by $q^{r}$ for some $r>0$.
Hence $\sigma$ can be extended to a representation $\wt\sigma_{0}$ of $W_{\sigma}\cap I_{F}$ such
that the generator $w$ of $\ZM$ acts by $^{w}\wt\sigma_{0}\simeq \wt\sigma_{0} \chi$ for some
character $\chi$ of $(W_{\sigma}\cap I_{F})/K_{F}$. But this character
admits a $(q^{r}-1)^{\rm th}$-root
$\chi_{0}$ (since $x\mapsto x^{q^{r}-1}$ is surjective on $\mu_{\ell^{\infty}}$ and
$\mu_{{p'}^{\infty}}$),
 so that $\wt\sigma_{0}\chi_{0}^{-1}$ is fixed by $w$, hence extends to a representation
$\wt\sigma$ as desired. 
\end{proof}

Let us apply this lemma to each $\sigma_{i}$ and pick an extension $\tilde\sigma_{i}$ to $W_{\sigma_{i}}$.
Then, putting
$$\hat\varphi:=\hat\varphi_{1}\oplus\cdots\oplus\hat\varphi_{r},\,\,\hbox{ with }
\hat\varphi_{i}=\oQl^{e}\otimes {\rm Ind}_{W_{\sigma_{i}}}^{W_{F}}(\tilde\sigma_{i}),$$
we get an extension $\varphi$ of $\phi$ such that
$$ (C_{\hat\mathbf{G}}(\phi),\alpha_{\varphi}) \simeq \prod_{i=1}^{r}
 {\rm Ind}_{W_{\sigma_{i}}}^{W_{F}}(\GL_{e_{i}})$$
where $W_{\sigma_{i}}$ acts trivially on $\GL_{e_{i}}$.
In particular $\alpha_{\varphi}$ fixes any diagonal \'epinglage of
$C_{\hat\mathbf{G}}(\phi)\simeq\prod_{i=1}^{r}\GL_{e_{i}}^{[W_{F}:W_{\sigma_{i}}]}$ and we
may identify $C_{\hat\mathbf{G}}(\phi)\rtimes_{\alpha_{\varphi}}W_{F}$ with the $L$-group $^{L}\mathbf{G}_{\phi}$.
Denoting by $F_{i}$ the finite extension such that $W_{F_{i}}=W_{\sigma_{i}}$, we then see that
 $\mathbf{G}_{\phi}\simeq \prod_{i=1}^{r} {\rm Res}_{F_{i}|F}(\GL_{e_{i}})$,
and that the factorization (\ref{factophi}) is  a unipotent factorization $\xi=\xi_{\varphi}$ of $\phi$ of the following form
$$ \xi:   
{^{L}\mathbf{G}_{\phi}}= 
\left(\prod_{i=1}^{r} {\rm Ind}_{W_{F_{i}}}^{W_{F}} \GL_{e_{i}}\right)\rtimes W_{F}
\To{}
\left(\prod_{i=1}^{r} \GL_{n_{i}} \right)\times W_{F} 
\injo \GL_{n}\times W_{F},$$
where $\xi_{|W_{F}}=\varphi$.

\begin{prop} \label{existunipfact}
Let $\phi$ be a $K_{F}$-parameter of a group $\mathbf{G}$  of $\GL$-type. Then 
$\mathbf{G}_{\phi}$ is also of $\GL$-type and $\phi$ admits a strict unipotent factorization
$\phi:\, K_{F}\To{1\times\id} {^{L}\mathbf{G}_{\phi}} \To\xi {^{L}\mathbf{G}}$.
\end{prop}
\begin{proof}
By definition of being of $\GL$-type,
  we may assume that $\mathbf{G}={\rm Res}_{F'|F}\mathbf{G'}$ for $\mathbf{G'}=\GL_{n}$ and $F'$ a
  finite separable extension of $F$. By Corollary \ref{factoShapiro}, we may assume that
  $\mathbf{G}=\GL_{n}$, which has just been treated.
\end{proof}

This proposition allows us to consider the following statement, that depends on an admissible parameter $\phi:\,
K_{F}\To{}{^{L}\mathbf{G}}$ (with $\mathbf{G}$ unspecified) :
\begin{center}
  $\UC_{F}(\phi)$ : 
  \begin{tabular}{|l}
for any  strict unipotent factorization $\xi:\,
{^{L}\mathbf{G}_{\phi}}\To{}{^{L}\mathbf{G}}$, \\
there is an equivalence of categories  $\Rep_{1}(G_{\phi})\simto\Rep_{\phi}(G)$  \\
  that extends the transfer map $\xi_{*}:\,\Irr{1}{G_{\phi}}\To{}\Irr{\phi}{G}$.
\end{tabular}
\end{center}
Note, that because of Remark \ref{Rkstrictunip}, replacing ``any'' by ``one'' in the first line
gives an equivalent statement. 
Again, we will denote by $\UC_F(\phi)^{-}$ the same statement without the compatibility with transfer.

\begin{lemme} \label{equivstatements}
  The following are equivalent.
  \begin{enumerate}
  \item Statement $\EC_F(\phi',\xi)$ is true for all $F$, $\phi'$ and $\xi$ satisfying the required conditions.
  \item Statement $\UC_F(\phi)$ is true for all $F$ and $\phi$.
  \item Statement $\UC_F(\phi)$ is true for all $F$ and $\phi$ pertaining to $\mathbf{G}=\GL_{n}$, and
    statement $\EC_F(1,\xi)$ is true for 
all base change $\xi:\, {^{L}\GL_{n}}\To{}{^{L}{\rm Res}_{F'|F}\GL_{n}}$, 
with $W_{F'}K_{F}=W_{F}$ (and $F$ allowed to vary).
  \end{enumerate}
Moreover the same equivalence holds for statements $\EC_F(\phi',\xi)^{-}$ and $\UC_F(\phi)^{-}$.
\end{lemme}
\begin{proof}
  $i)\Rightarrow iii)$ is clear. 
To prove  $ii)\Rightarrow i)$, start with $(\phi',\xi)$, choose  a strict unipotent factorization
$\xi'$ of $\phi'$  and consider the diagram
$$\phi:\, K_{F} \To{1} {^{L}\mathbf{G}'_{\phi'}}\To{\xi'}{^{L}\mathbf{G}'}\To{\xi}{^{L}\mathbf{G}}.$$
Then $\xi\circ\xi'$ is a unipotent factorization of $\phi$, albeit not strict a priori.
By Proposition \ref{propunipfacto},
 it is equivalent to the composition $\xi''\circ\alpha$ of a strict unipotent factorization and a
$W_{F}$-invariant outer automorphism $\alpha$ of $\hat\mathbf{G}_{\phi}$. A $W_{F}$-invariant outer automorphism of
$\hat\mathbf{G}_{\phi}$ induces an $F$-automorphism of $\mathbf{G}_{\phi}$, well-defined up to
$G_{\phi}$-conjugacy, hence an endo-equivalence of categories of $\Rep(G_{\phi})$ and in particular
of $\Rep_{1}(G_{\phi})$ (since the trivial representation is fixed). By \cite[Prop. 5.2.5]{Hainesstable}, 
this equivalence is known to be compatible with Langlands' transfer.
Therefore, using this equivalence and the ones granted by $\UC_F(\phi')$ and $\UC_F(\phi)$,
we get $\EC(\phi',\xi)$.

Let us prove $iii)\Rightarrow ii)$. We want to check $\UC_F(\phi)$ for any $\phi$. It is sufficient to
do so when $\mathbf{G}={\rm Res}_{F'|F}(\GL_{n})$. Let $\xi$ be a strict unipotent factorization of
$\phi$. We have a factorization $\xi=\tilde{\xi'}\circ\xi_{u}$ as in Corollary \ref{factoShapiro} ii). By
hypothesis, and thanks to Remark \ref{RkShapiro}, we can find an equivalence of categories associated to $\tilde{\xi'}$,
so we are left with finding one associated to $\xi_{u}$. With the notation of Corollary \ref{factoShapiro} i), we
 have a further factorization of $\xi_{u}$ :
$$\xi_{u}: {^{L}\mathbf{G}_{\phi}} \simto {^{L}{\rm Res}_{F''|F}(\mathbf{G}''_{\phi''})}
 \To{}{^{L}{\rm Res}_{F'|F}(\mathbf{G}'_{\phi'})}$$
which shows that it is sufficient to do it when $F''=F$, \emph{i.e.} $W_{F'}K_{F}=W_{F}$. In this
case, $\xi_{u}$ is a base change $L$-homomorphism 
$ \xi_{u}:\, {^{L}\mathbf{G}_{\phi}} \To{} {^{L}{\rm Res}_{F'|F}(\mathbf{G}_{\phi}\times_{F}F')}.$
Now, $\mathbf{G}_{\phi}$ is of $\GL$-type and ``$K_{F}$-unramified'',
in the sense that it splits over an extension $F_{0}$ of $F$ such that $K_{F_{0}}=K_{F}$.
So we need an equivalence associated to a base change homomorphism of the form
$$ \xi_{u}:\, {^{L}({\rm Res}_{F_{0}|F}\GL_{n})} \To{} {^{L}{\rm Res}_{F'|F}({\rm
    Res}_{F_{0}|F}\GL_{n}\times_{F}F')}$$
where $F_{0}$ is a $K_{F}$-unramified extension of $F$.
But then $F'$ and $F_{0}$ are disjoint, so we have 
${\rm Res}_{F_{0}|F}\GL_{n}\times_{F}F' = {\rm Res}_{F'F_{0}|F'}\GL_{n}$ and the above
$L$-homomorphism takes the form
$$  {^{L}({\rm Res}_{F_{0}|F}\GL_{n})} \To{} {^{L}{\rm Res}_{F_{0}|F}({\rm
    Res}_{F_{0}F'|F_{0}}\GL_{n}})$$
and is thus ``induced'' from the base change $L$-homomorphism over $F_{0}$
$$  {^{L}(\GL_{n})} \To{} {^{L}{\rm   Res}_{F_{0}F'|F_{0}}\GL_{n}}. $$
Using Remark \ref{RkShapiro} again, it is enough to associate an equivalence to the latter $L$-homomorphism, but
this is precisely part of the hypothesis in iii).
\end{proof}

\alin{Proof of Theorem \ref{transferql}} Here we assume that $K_{F}=I_{F}$ and
we will prove that $\EC_F(\phi',\xi)$ holds true for all $\phi'$ and
$\xi$ that  satisfy the
required conditions. It is enough to prove the statements in point iii) of the previous
lemma. We will denote by $\HC(q,n)$ the extended Iwahori-Hecke algebra of type $A_{n-1}$
with parameter $q$ over $\oQl$. We will also denote by $q_{F}$ the cardinality of the
residue field of $F$.

\emph{Totally ramified base change of the unipotent block of $\GL_{n}$.}
We use Zelevinski's classification $m\mapsto Z(m)$ 
of $\Irr{\oQl}{\GL_{n}(F)}$ in terms of multisegments of
 unramified characters of $F^{\times}$.
Let $F'|F$ be a totally ramified extension. The base change for $\GL_{1}$ is induced
by the norm map $(F')^{\times}\To{} F^{\times}$. Since the latter induces an isomorphism
${F'}^{\times}/\OC_{F'}^{\times} \simto F^{\times}/\OC_{F}^{\times}$, the base change is a
bijection on unramified characters, hence also on multisegments  $m\mapsto m'$.
 Since base change is compatible with parabolic induction, and thus with the Langlands
 quotient construction, it is also compatible with the Zelevinski construction, in the
 sense that the base change of $Z(m)$ has to be $Z(m')$. 
Now,  by a theorem of
Borel, there is a natural equivalence of categories between $\Mod(\HC(q_{F},n))$ and
$\Rep_{1}(\GL_{n}(F))$ which takes Rogawski's classification of simple modules
of $\HC(q_{F},n)$ in terms of multisegments of characters of $\ZM^{n}$
 to Zelevinski's classification, see \cite{Rogawski}.
The desired equivalence between $\Rep_{1}(\GL_{n}(F))$ and $\Rep_{1}(\GL_{n}(F'))$ hence
follows from the equality
 $q_{F'}=q_{F}$.

\emph{Property $\UC_F(\phi)$ for $\GL_{n}$.} As in paragraph \ref{computGL}, let us write
$\phi=\hat\phi\times \id_{I_{F}}$ 
and decompose  $\hat\phi=\hat\phi_{1}\oplus\cdots\oplus\hat\phi_{r}$. If $r>1$, paragraph
\ref{computGL} tells us that any unipotent factorization of $\phi$ factors through a Levi
embedding that induces an isomorphism of centralizers. Thanks to Example \ref{exempleequiv},
 we may thus assume $r=1$.
In this case, let $\sigma$ be an irreducible constituent of $\hat\phi$. Its stabilizer
$W_{\sigma}$ in $W_{F}$ is the Weil group $W_{F_{f}}$ of ``the'' unramified extension of
degree $f=[W_{F}:W_{\sigma}]$ and paragraph \ref{computGL} tells us that 
$\mathbf{G}_{\phi}\simeq {\rm  Res}_{F_{f}|F}(\GL_{e})$ where $e=n/(f\dim\sigma)$.  
Pick an extension $\tilde\sigma$ of $\sigma$ to $W_{F_{f}}$
and put 
$\hat\varphi:=\oQl^{e}\otimes{\rm Ind}_{W_{F_{f}}}^{W_{F}}(\tilde\sigma)$. We get an extension $\varphi$ of $\phi$ to $W_{F}$,
whose associated
strict unipotent factorization $\xi_{\varphi}$ 
has the following effect on parameters. Identify $\Phi(\mathbf{G},\oQl)$, 
resp. $\Phi(\mathbf{G}_{\phi},\oQl)$,
with the set of (classes of)
Frobenius-semisimple  continuous $\oQl$-representations of $W_{F}'$ of dimension $n$, 
resp. of $W_{F_{f}}'$ of dimension $e$. 
Then the transfer
map $\xi_{\varphi,*}$  is given by
$$ 
 \rho\in\Phi(\mathbf{G}_{\phi},\oQl)
\mapsto \cind{W_{F_{f}}'}{W_{F}'}{\tilde\sigma\otimes\rho} \in \Phi(\mathbf{G},\oQl),$$
For example, denoting by ${\rm Sp}_{e}$  the special representation of dimension $e$ 
(associated to the Steinberg representation), we see that 
$\xi_{\varphi,*}({\rm Sp}_{e})={\rm  Sp}_{e}\otimes {\rm Ind}_{W_{F_{f}}}^{W_{F}}(\tilde\sigma)$. Let us translate this in terms of
irreducible representations. Let $\pi$ be the supercuspidal representation of 
$\GL_{f{\rm  dim}(\sigma)}(F)$ that corresponds to $ {\rm Ind}_{W_{F_{f}}}^{W_{F}}(\tilde\sigma)$ via the LLC. To any pair $(\chi,a)$ with
$\chi$ an unramified character of $F_{f}^{\times}$ and $a\in \NM$, we associate the segment
$\Delta_{\pi}(\chi,a)=(\pi_{\chi},\pi_{\chi}\otimes\nu,\cdots,\pi_{\chi}\otimes\nu^{a-1})$ where
$\nu=|\det|_{F}$ and $\pi_{\chi}=\pi\otimes(\chi^{1/f}\circ\det)$ (which is independent of the choice
of an $f^{\rm th}$-root of $\chi$). This extends to a bijection
$m\mapsto m_{\pi}$ between  multisegments of unramified characters of $F_{f}^{\times}$ and
 ``multisegments of type $\pi$''. Then the formula above shows that the transfer map
 $\xi_{\varphi,*}:\Irr 1 {G_{\phi}}\To{}\Irr\phi G$ takes $Z(m)$ to $Z(m_\pi)$ in Zelevinski's
 notation (or equivalently $L(m)$ to $L(m_{\pi})$), compare \cite[\S 2]{HeSMF}.


Now let us put $\HC_{\phi}:=\HC(q_{F}^{f},e)$.
Thanks to their theory of simple types, 
Bushnell and Kutzko produce ``natural'' equivalence\emph{s} of categories between 
$\Mod(\HC_{\phi})$ and $\Rep_{\phi}(\GL_{n}(F))$ 
\cite[Thm 7.5.7]{BK}. These equivalences are  unramified 
twists of each other \cite[Prop. 7.5.10]{BK}.
  In light of the above, we will normalize the equivalence so that
 it takes the sign character of $\HC_{\phi}$ to the ``generalized'' Steinberg representation 
${\rm  St}_{e}(\pi)$. Then, the compatibility of Bushnell-Kutzko equivalences with normalized
parabolic induction \cite[Thm. 7.6.1]{BK}
and unramified twisting \cite[Prop. 7.5.12]{BK} shows that  it also takes the simple module $M(m)$ associated to the multisegment $m$
to $Z(m_{\pi})$.
On the other hand, as recalled previously, Borel's theorem produces a ``canonical'' equivalence of categories between
$\Mod(\HC_{\phi})$ and $\Rep_{1}(G_{\phi})$, that takes $M(m)$ to $Z(m)$.
By composition we get an   equivalence between $\Rep_{1}(G_{\phi})$ and
$\Rep_{\phi}(\GL_{n}(F))$ that takes $Z(m)$ to $Z(m_{\pi})$, as desired.

\begin{rak}
  We may ask whether an equivalence as in statement $\EC_F(\phi',\xi)$ is unique. In view of the above
  discussion, this reduces to asking whether an auto-equivalence $\alpha$ of $\Mod(\HC(q,n))$ that
  ``preserves simple modules'' (in the sense that $\alpha(M)\simeq M$ for each simple module $M$) is
  isomorphic to the identity functor. For this, one has to compute the Picard group of $\HC(q,n)$
  over its center.
\end{rak}

\alin{``Proof'' of Theorem \ref{transferzltame}} \label{prooftransferzl}
Here we explain how Theorem \ref{transferzltame} follows from constructions in \cite{Datequiv}.
So we assume that $K_{F}=I_{F}^{(\ell)}$
and we consider statements $\EC_F(\phi',\xi)$ for $\oZl$-blocks when both $\phi'$ and $\xi$ are \emph{tame.} Recall
that this means that $\phi'_{|P_{F}}$, resp. $\xi_{|P_{F}}$,
 is equivalent to the trivial parameter. We note that
Lemma \ref{equivstatements}
 remains true if we impose tameness of $\phi'$, $\xi$ and $\phi$ in each item i), ii) or
iii). This is because a unipotent factorization of a tame parameter is tame, and
 Shapiro bijections preserve tameness. In this tame setting we can
 reduce further these statements as follows.
 \begin{lem}
Assertions i), ii) and iii) of Lemma \ref{equivstatements} restricted to tame parameters are equivalent to :
     \begin{enumerate}
     \item[iv)] Statement $\EC_F(\phi', \xi)$ is true in the following
       cases :
       \begin{enumerate}
       \item $\phi'$ is tame and $\xi$ is an unramified automorphic
         induction $^{L}{\rm
           Res}_{F_{f}|F}(\GL_{n/f})\To{}{^{L}\GL_{n}}$
       \item $\phi'=1$ and $\xi$ is a totally $\ell'$-ramified base
         change $ {^{L}\GL_{n}}\To{}{^{L}{\rm Res}_{F'|F}\GL_{n}}$.
       \end{enumerate}
     \end{enumerate}
Moreover the same equivalence holds for statements $\EC_F(\phi',\xi)^{-}$ and $\UC_F(\phi)^{-}$.
 \end{lem}
 \begin{proof}
  $i) \Rightarrow iv)$ is clear, so we only need to check that $iv)\Rightarrow
   iii)$, and in fact it is sufficient to prove that $iv)(a)$ implies property $\UC_F(\phi)$
   for tame parameters $\phi$ of $\mathbf{G}=\GL_{n}$.  Write $\phi=\hat\phi\times
   \id_{I_{F}^{(\ell)}}$ with $\hat\phi$ an $n$-dimensional representation of
   $I_{F}^{(\ell)}/I_{F}$. By \ref{computGL}, we know that if $\mathbf{G}_{\phi}$ is not
   quasi-simple then any unipotent factorization factors through some Levi
   subgroup. Thanks to Example \ref{exempleequiv}, we may thus assume that
   $\mathbf{G}_{\phi}$ is quasi-simple. In this case, \ref{computGL} tells us that we can
   extend $\hat\phi$ to $\hat\varphi\simeq \oQl^{e}\otimes {\rm
     Ind}_{W_{\sigma}}^{W_{F}}(\tilde\sigma)$, where $\sigma$ is an irreducible summand of
   $\hat\phi$, and $\tilde\sigma$ is an extension of $\sigma$ to its normalizer
   $W_{\sigma}$ in $W_{F}$. Factorization (\ref{factophi}) of $\varphi$ then reads
$$ \varphi:\, W_{F}\To{(1,\id_{W})} 
{^{L}\mathbf{G}_{\phi}}= C_{\hat\mathbf{G}}(\phi)
\rtimes_{\alpha_{\varphi}} W_{F} \To{\id.\varphi} \GL_{n}\times W_{F},$$
with $\mathbf{G}_{\phi}= {\rm Res}_{F_{\sigma}|F}(\GL_{e})$ for the finite extension
$F_{\sigma}$ such that  $W_{F_{\sigma}}=W_{\sigma}$.
In our tame context, since $I_{F}/P_{F}$ is abelian, $W_{\sigma}$ contains $I_{F}$ hence
$F_{\sigma}=F_{f}$ is the unramified extension of some degree $f$ over $F$.  Moreover,
since $\sigma$ has dimension $1$, we have  $f=n/e$ and $C_{\hat\mathbf{G}}(\phi)$ identifies with the diagonal
Levi subgroup $(\GL_{e})^{f}$ of $\GL_{n}$.

Now, looking at the unipotent factorization above, we see that it involves the ``right groups''
 ${\rm Res}_{F_{f}|F}(\GL_{e})$ and $\GL_{n}$, but the morphism of $L$-groups $\id.\varphi$ is \emph{not} the
automorphic induction morphism. So we look for another factorization (not unipotent) of
$\varphi$,  involving the same groups but with the automorphic induction morphism.

We know that, as with any extension of $\phi$,  the group $\hat\varphi(W_{F})$ is
contained in the normalizer $\NC=\NC_{\hat\mathbf{G}}(C_{\hat\mathbf{G}}(\phi))$. For
formal reasons, $\varphi(W_{F})$ is therefore contained in the subgroup
$\NC\times_{\NC/\NC^{\circ}}W_{F}$ of $^{L}\mathbf{G}$, where the fibered product is for
the composition $W_{F}\To{\hat\varphi}\NC\To{}\NC/\NC^{\circ}$.
Actually, if $\varepsilon$ denotes any diagonal \'epinglage of 
$C_{\hat\mathbf{G}}(\phi)\simeq (\GL_{e})^{f}$, we know that $\hat\varphi(W_{F})$ is contained
in the normalizer $\NC_{\varepsilon}$ of this \'epinglage in $\NC$, and therefore
$\varphi(W_{F})\subset \NC_{\varepsilon}\times_{\NC/\NC^{\circ}}W_{F}$.

Now the point is that $\NC^{\circ}=C_{\hat\mathbf{G}}(\phi)$ (because it is a Levi
subgroup), $\NC_{\varepsilon}^{\circ}=Z(C_{\hat\mathbf{G}}(\phi))$, 
and we can find  a section morphism 
$\NC_{\varepsilon}/\NC_{\varepsilon}^{\circ}= \NC/\NC^{\circ} \injo \NC_{\varepsilon}$
(taking permutation matrices).
Then the action $\alpha_{\varepsilon}$ of $W_{F}$ on $C_{\hat\mathbf{G}}(\phi)$ through $\NC/\NC^{\circ}$
makes the semi-direct product $C_{\hat\mathbf{G}}(\phi)\rtimes_{\alpha_{\varepsilon}}
W_{F}$ isomorphic to $^{L}\mathbf{G}_{\phi}$, and identifies the semi-direct product
$Z(C_{\hat\mathbf{G}}(\phi))\rtimes_{\alpha_{\varepsilon}} W_{F}$  with
$Z(\hat\mathbf{G}_{\phi})\rtimes  W_{F}\subset {^{L}\mathbf{G}}_{\phi}$.
Moreover this section induces an isomorphism 
$C_{\hat\mathbf{G}}(\phi)\rtimes (\NC/\NC^{\circ})\simto \NC$, which in turn induces
an isomorphism
$$\iota:\, ^{L}\mathbf{G}_{\phi}= C_{\hat\mathbf{G}}(\phi)\rtimes_{\alpha} W_{F} =
(C_{\hat\mathbf{G}}(\phi)\rtimes (\NC/\NC^{\circ}))\times_{\NC/\NC^{\circ}}W_{F}
\simto \NC\times_{\NC/\NC^{\circ}}W_{F}.$$
Now consider $\varphi':=\iota^{-1}\circ\varphi : W_{F}\To{}{^{L}\mathbf{G}_{\phi}}$ and
$\xi$ the composition of $\iota$ with the inclusion of $\NC\times_{\NC/\NC^{\circ}}W_{F}$
into $^{L}\mathbf{G}$. By construction, we have a factorization
$$ \varphi:\, W_{F}\To{\varphi'}{^{L}\mathbf{G}_{\phi}}={^{L}{\rm Res}_{F_{f}|F}\GL_{e}} \To{\xi} {^{L}\mathbf{G}}$$
and $\xi$ is the automorphic induction $L$-morphism.
Restricting to $I_{F}^{(\ell)}$, we get a factorization
$$ \phi:\, I_{F}^{(\ell)}\To{\phi'} {^{L}\mathbf{G}_{\phi}}=
{^{L}{\rm Res}_{F_{f}|F}\GL_{e}} \To{\xi} {^{L}\mathbf{G}}$$ where $\phi'$ is
extendable to $W_{F}$ (namely to $\varphi'$) and $\xi$  induces an isomorphism on
centralizers. 

Now, hypothesis
iv)(a) provides us with an equivalence of categories
$\Rep_{\phi'}(\mathbf{G}_{\phi})\simto\Rep_{\phi}(\mathbf{G})$, 
but  what we need is  an equivalence $\Rep_{1}(\mathbf{G}_{\phi})\simto\Rep_{\phi}(\mathbf{G})$.
For this, observe that $\varphi'$ factors through $Z(\hat\mathbf{G}_{\phi})\rtimes W_{F}$
(because $\hat\varphi(W_{F})\subset \NC_{\varepsilon}$). Therefore
$\varphi'$  corresponds to a character  $G_{\phi}\To{\chi'}\oQl^{\times}$ as in 
\cite[10.2]{BorelCorvallis}. Twisting by this character provides 
an equivalence $\Rep_{1}(\mathbf{G}_{\phi})\simto
\Rep_{\phi'}(\mathbf{G}_{\phi})$ and composing with the latter gives the desired correspondence.
\end{proof}

We close this paragraph by stating that the weakened forms of iv)(a) and iv)(b) (i.e. whithout
compatibility with transfer) are proved in \cite{Datequiv}. What is missing at the moment to get the
strong form is the compatibility of the construction in loc. cit. with parabolic induction.


\end{document}